\documentclass[11pt,twoside]{amsart}
\usepackage{amsfonts}
\usepackage{amsmath}
\usepackage{amssymb}
\usepackage[latin1]{inputenc} % pour les accons
\usepackage{mathrsfs}
\usepackage{color}
\newtheorem {thm}{Theorem}
\newtheorem {pro}{Proposition}
\newtheorem {cor}{Corollary}
\newtheorem {lem}{Lemma}
\newtheorem {defn}{Definition}

\newtheorem {rem}{\it Remark}

\newtheorem {exe}{\it Example}

\numberwithin{equation}{section}

\def\Supp{\mathop{\rm Supp}}
\newcommand{\cal}       {\mathcal}
\newenvironment{proof of}{\noindent}{\it Proofof}

\def\finpr{\hfill \hbox{
\vrule height 1.453ex  width 0.093ex  depth 0ex \vrule height
1.5ex  width 1.3ex  depth -1.407ex\kern-0.1ex \vrule height
1.453ex  width 0.093ex  depth 0ex\kern-1.35ex \vrule height
0.093ex  width 1.3ex  depth 0ex}}
\newenvironment{proofof}{{\noindent {\it Proof of }}}{\hfill \finpr \\ }
\def\cb{{\Bbb C}}
\def\rb{{\Bbb R}}

                    \def\b{\beta}
                      
          \def\L{\Lambda}         \def\l{\lambda}
\def\nb{\mathbb N}       \def\ds{\displaystyle}  
           \let\L=\longrightarrow  
\def\v{\varphi}                  \let\l=\rightarrow
               
                 \def\ds{\displaystyle}

\def\1{1\!\rm l}

\title{Complex Hessian Operator associated to an $m$-positive closed current and weighted $m$-capacity}
\author[ Hadhami Elaini and Fredj Elkhadhra]{ Hadhami Elaini and Fredj Elkhadhra}
\address{ University of Sousse\\ Higher School of Sciences and Technology of Hammam Sousse\\ MaPSFA (LR 11 ES 35)\\ 4011 Hammam Sousse\\ Tunisia.}
\email{elaynihadhami@gmail.com, fredj.elkhadhra@essths.u-sousse.tn}
\subjclass[Mathematics Subject Classification.]{32W50, 32C30, 31A15}
\begin{document}
\begin{abstract} In this paper, we first study the definition and the continuity of the complex Hessian operator associated to an $m$-positive closed current $T$, for some classes of unbounded $m$-subharmonic functions as well as when we consider a regularization sequence of $T$. Next, we introduce the notion of weighted $(m,T)$-capacity in the complex Hessian setting and we investigate the link with the weighted $m$-extremal function. As an application we give a characterization of the Cegrell classes ${\mathscr F}^m$ and ${\mathscr E}^m$ by means of the weighted $(m,1)$-capacity. Furthermore, we prove a subsolution theorem for a general complex Hessian equation relatively to $T$.
\end{abstract} \maketitle
\tableofcontents
\section{Introduction}
Let $\Omega$ be a bounded open subset of $\cb^n$ and $m$ an integer such that $1\leqslant m\leqslant n$. Denote by ${\cal {SH}}_m(\Omega)$ the class of $m$-subharmonic ($m$-sh) functions on $\Omega$. In the borders cases $m=1$ and $m=n$, correspond to subharmonic functions and plurisubharmonic (psh) functions respectively. In this paper, we consider the notion of $ m$-positivity of forms and currents introduced firstly by Blocki \cite{Bl} and next by Lu \cite{lu article} and by Dhouib-Elkhadhra \cite{ref9}. Denote by ${\mathscr C}_p^m(\Omega)$ the convex cone of $m$-positive closed currents of bidegree $(p,p)$ on $\Omega$ (see Definition 1), ${\rm Supp}T$ is the support of a given $m$-positive current $T$ and $\beta=dd^c|z|^2=\frac{i}{\pi}\partial\overline\partial|z|^2$. Beside the introduction the paper has four sections. In Section 2, we recall some basic notations and definitions necessarily for the rest of this paper. Section 3 is reserved to the study of  the definition and the continuity of the complex  Hessian operator $(dd^c.)^{m-p}\wedge\beta^{n-m}\wedge T$.
Recall that the definition and the continuity of such operator have been studied intensively in the past ten years. In 2012, Lu \cite{ref13} and Sadullaev-Abdullaev \cite{ref17} have proved the continuity of the complex Hessian operator for decreasing sequences of locally bounded $m$-sh functions. Recently, in 2016, Dhouib-Elkhadhra \cite{ref9} and Wan-Wang \cite{ref19} have obtained the same result for $m$-sh functions which are bounded near $\partial\Omega\cap\Supp T$ as well as for $m$-sh functions which their unbounded locus have a small Hausdorff measure. Building on a very recently work of \cite{ref1} on the complex Monge-Amp\`ere operator, we improve the later result, by assuming that one of the considered $m$-sh functions is integrable with respect to the trace measure of $T$. Next, by letting ourselves be inspired by a technics go back to Ben Messaoud-El Mir \cite{me-el} in the complex setting, we investigate the continuity of the operator $(dd^{c}u_{1}^{j}\wedge...\wedge dd^{c}u_{q}^{j}\wedge\beta^{n-m}\wedge T_{j})_{j}$ where $T_{j}=T\star\chi_{j}$ is the standard regularization by convolution of $T$ and $(u_{k}^{j})_j$ are sequences of monotone decreasing $m$-sh functions not necessarily bounded. In the remaining of Section 3, we study how we can replace the condition of  monotone increasing sequence of $m$-sh functions by a kind of convergence in capacity. Roughly speaking, until the work of Xing \cite{ref21}, convergence in capacity appears as an effective tools in studying continuity of complex Monge-Amp\`ere operator and complex Hessian operator. In fact, the monotonic decreasing condition can be relaxed to a convergence in the sense of capacity (see \cite{ref21}, \cite{ref9}). We finish Section 3, by showing that any increasing sequence of $m$-sh functions, converges in the sense of the capacity $cap_{m-1,T}$ (see Subsection 3.2), provided that the starting function is integrable with respect to the trace measure of $T$. As a consequence, we deduce that  the complex  Hessian operator $(dd^c.)^{m-p}\wedge\beta^{n-m}\wedge T,$ is continuous on locally uniformly bounded sequences of $m$-sh functions which are monotonically increasing a way from a vanishing $cap_{m,T}$ subset, where $cap_{m,T}$ is the $(m,T)$-capacity introduced by \cite{ref9}. We recover then a result of Bedford-Taylor \cite{ref2} (see also  \cite{ref6}) for the complex case $m=n$ and the trivial current $T=1$. In Section 4, starting from a given negative $m$-sh function $u$ and an $m$-positive closed current $T$, we introduce  the $(m,T)$-capacity with weight $u$ denoted by $cap_{m,T,u}$ which generalize the $(m,T)$-capacity $cap_{m,T}$ for the constant weight $u=-1$. Especially when $T=1$, Nguyen V.T \cite{Van Nguyen} confirms that the weighted $m$-capacity serves in the characterization of ${\mathscr E}^m(\Omega)$ the generalized Cegrell class introduced and studied by \cite{ref13}. Among other interesting related properties, if $u\in{\mathscr E}^m(\Omega)$ and $cap_{m,u}=cap_{m,1,u}$ then we prove that there is a strong link between the weighted $(m,1)$-capacity and the weighted $m$-extremal function $h_{m,E,u}^*$ (see Definition 3). Namely, we establish the following estimates: $cap_{m,u}^\ast\left({\ds\mathop  E^\circ}\right)\leqslant (dd^ch_{m,E,u}^*)^m\wedge\beta^{n-m}(\overline E)\leqslant cap_{m,u}^\ast(\overline E)$ for any set $E\Subset\Omega$, where $cap_{m,u}^\ast$ is the outer capacity associated to $cap_{m,u}$. In particular, we recover the estimates obtained by \cite{ref8} for the border case $m=n$. Furthermore, we show how $m$-weighted capacity can be used to  characterize the interesting Cegrell subclass ${\mathscr F}^m(\Omega)\subset{\mathscr E}^m(\Omega)$ and which also provides several examples of functions in this subclass. Next, similarly as in \cite{ref18}, if $\Omega$ is $m$-strongly pseudoconvex, we present the Sadullaev weighted $m$-capacity by means of the defining function of $\Omega$ and we close Section 4 by establishing that such capacity can be compared with $cap_{m,u}^\ast$ for adequate conditions on the weight. The results of Section 4 generalize the one obtained by \cite{ref15} for the complex setting as well as the comparison of capacities proved by \cite{ref17} when the weight is the constant $m$-sh function $-1$. Finally, in Section 5 we deal with the solution of the following generalized complex Hessian equation $(dd^c.)^{m-p}\wedge\beta^{n-m}\wedge T=\mu$, for a given current $T\in{\mathscr C}_p^m(\Omega)$ and a positive measure $\mu$. It should be mention here that especially in the case $T=1$, resolution of the complex Hessian equation is crucial in the development of the complex Hessian theory. By an adaptation of a technics due to Xing \cite{ref21} and by using some properties from the potential theoretic aspect of $m$-sh functions like the generalized Xing-type comparison principle inequality, we obtain a subsolution result for the above general equation. More precisely, we give sufficient conditions on the singularities of $T$ guaranteeing that a subsolution of the generalized complex Hessian equation yields in fat a solution. We extend then the main work of Xing \cite{ref21} for the trivial current $T=1$ and the complex setting $m=n$.
\section{Preliminaries}
 In this section, we recall some elements of complex Hessian theory that will be used throughout this paper. According to Blocki \cite{Bl}, a real $(1,1)$-form $\alpha$ is said $m$-positive on $\Omega$ if at every point of $\Omega$ we have $\alpha^j\wedge\beta^{n-j}\geqslant 0$, $j=1,...,m$. The following lemma of Blocki \cite{Bl} is crucial in this concept of $m$-positivity:
\begin{lem} Let $1\leqslant p\leqslant m$. If $\alpha_{1},...,\alpha_p$ are $m$-positive $(1,1)$-forms then $\alpha_1\wedge...\wedge\alpha_p\wedge\b^{n-m}\geqslant 0$.
\end{lem}
 According to \cite{ref9}, we have:
\begin{defn} A current $T$ of bidegree $(p,p)$ on $\Omega$ such that $p\leqslant m\leqslant n$ is said $m$-positive if for any $m$-positive $(1,1)$-forms $\alpha_1,...,\alpha_{m-p},$ we have
$\alpha_1\wedge...\wedge\alpha_{m-p}\wedge\beta^{n-m}\wedge T\geqslant 0.$
\end{defn}
We should be noted here that in Definition 1 we can consider $m$-positive forms with constant coefficients.  Indeed, if $\alpha=\beta^{n-m}\wedge\alpha_1\wedge...\wedge\alpha_{m-p},$ where $\alpha_1,...,\alpha_{m-p}$ are general $m$-positive $(1,1)$-forms, then, for all $z_0\in \Omega$ there exists $r>0$ such that $\alpha(z)=\alpha(z_0)+ O(|z-z_0|)$ on the euclidean ball $\mathscr{B}(z_0,r)$. For $\varepsilon>0$, we have $$\int_{\mathscr{B}(z_0,r)}T\wedge(\alpha+\varepsilon\beta^{n-p})=\int_{\mathscr{B}(z_0,r)}T\wedge \alpha(z_0)+\varepsilon\int_{\mathscr{B}(z_0,r)}T\wedge\beta^{n-p}+\int_{\mathscr{B}(z_0,r)}T\wedge O(|z-z_0|).$$  By positivity of the current $T\wedge\beta^{n-m}$, if $T\wedge\beta^{n-p}=0$, then $T\wedge\beta^{n-m}=0$ and therefore $T\wedge\alpha=0$. If $T\wedge\beta^{n-p}>0$, then since the last integral is sufficiently small when $r\ll 1$, we see that $\int_{\mathscr{B}(z_0,r)}T\wedge(\alpha+\varepsilon\beta^{n-p})\geqslant 0$  and therefore we deduce that $T\wedge\alpha\geqslant 0$ by letting $\varepsilon\l 0$. Now, we recall briefly the notion of $m$-subharmonic ($m$-sh) functions. A function $u: \Omega\rightarrow \rb$ is said $m$-sh if it is subharmonic
and if  $dd^{c}u\wedge\alpha_{1}\wedge...\wedge\alpha_{m-1}\wedge \beta^{n-m}\geqslant 0 $ for every $ \alpha_{1},...,\alpha_{m-1}$ $m$-positive $(1,1)$-forms. For convenience, we will denote by ${\cal {SH}}_{m}(\Omega)$ the set of $m$-sh functions on $\Omega$. In the following, we list the most basic properties of ${\cal {SH}}_{m}(\Omega)$ that we shall need later on.
\begin{pro} \
\begin{enumerate}
\item If $u$ is of class $C^2$ then $u\in{\cal {SH}}_m(\Omega)$ if and only if $dd^cu$ is $m$-positive on $\Omega$.
\item ${\cal {PSH}}(\Omega)={\cal {SH}}_n(\Omega)\subset{\cal {SH}}_{n-1}(\Omega)\subset\cdots\subset{\cal {SH}}_1(\Omega)={\cal {SH}}(\Omega):=\{u,\ {\rm subharmonic\ on}\ \Omega\}$.
\item If $u\in{\cal {SH}}_m(\Omega)$, then the standard regularization $u_j=u\star\chi_j\in{\cal {SH}}_m(\Omega_j)\cap{\mathscr C}^\infty(\Omega_j)$, where $\Omega_j=\{x\in\Omega:\ d(x,\partial\Omega)>1/j\}$. Moreover, $(u_j)_j$ decreases pointwise to $u$.
\item Let $u,v\in{\cal {SH}}_m(\Omega)$ then $\max(u,v)\in{\cal {SH}}_m(\Omega)$.
\item If $(u_\alpha)_\alpha\Subset{\cal {SH}}_m(\Omega)$, $u=\sup_\alpha u_\alpha<+\infty$ and $u$ is upper semicontinuous then $u$ is $m$-sh.
\end{enumerate}
\end{pro}
Assume that $T\in {\mathscr C}_{p}^{m}(\Omega)$ and $u_{1},...,u_{k}\  (k\leqslant m-p) $ are $m$-sh locally bounded functions  on $ \Omega$. Then the complex Hessian operator can be defined inductively by setting:  $$dd^{c}u_{1}\wedge...\wedge dd^{c}u_{k}\wedge \beta^{n-m} \wedge T=dd^{c}(u_{1} dd^{c}u_{2}\wedge...\wedge dd^{c}u_{k}\wedge\beta^{n-m}\wedge T) $$  which is a closed and $m$-positive current in the sense of Lu. Based on Definition 1 and by rewriting the proof of Lu \cite{ref13}, we get
\begin{pro}
Let $ T\in{\mathscr C}_p^m(\Omega)$, $ u_{0},...,u_{k}\in {\cal {SH}}_{m}(\Omega) \cap L_{loc}^{\infty}(\Omega),\ k\leqslant m-p$ and $ u_{0}^{j},...,u_{k}^{j} $ are  decreasing sequences of  $m$-sh functions which converge respectively to $ u_{0},...,u_{k}$. Then, we have:
\begin{enumerate}
\item $ u_{0}^{j}dd^{c}u_{1}^{j} \wedge...\wedge dd^{c}u_{k}^{j}\wedge\beta^{n-m}\wedge T\L u_{0}dd^{c}u_{1} \wedge...\wedge dd^{c}u_{k}\wedge\beta^{n-m}\wedge T$ weakly on $\Omega$.
\item $dd^{c}u_{0}^{j} \wedge...\wedge dd^{c}u_{k}^{j}\wedge\beta^{n-m}\wedge T \L dd^{c}u_{0}\wedge...\wedge dd^{c}u_{k}\wedge\beta^{n-m}\wedge T$ weakly on $\Omega$, $k<m-p$.
\end{enumerate}
\end{pro}
Analogously to the famous Bedford-Taylor capacity \cite{ref2}, Dhouib-Elkhadhra \cite{ref9} have associated to any current $T\in{\mathscr C}_p^m(\Omega)$ the $(m,T)$-capacity of a compact $K\subset\Omega$  by setting:
\vskip0.1cm
  $cap_{m,T}(K,\Omega)=cap_{m,T}(K)=\sup\{\int_{K}(dd^{c}u)^{m-p}\wedge\beta^{n-m}\wedge T,\  u\in {\cal {SH}}_{m}(\Omega) \ and\  -1\leqslant u\leqslant 0 \} .$
\vskip0.1cm For all $E\subset\Omega$,  we have  $ cap_{m,T}(E,\Omega)=\sup \left\lbrace cap_{m,T}(K,\Omega)\ with \  K\  \emph{is a compact of }E   \right\rbrace .$
\vskip0.15cm
 When $T=1$, we recover the $m$-capacity defined by Lu \cite{ref13}. It was proved by \cite{ref9} that every  bounded $m$-sh function is quasicontinuous with respect to the $(m,T)$-capacity. A set $ A\subset\Omega $ is called $(m,T)$-pluripolar if $ cap_{m,T}(A,\Omega)=0$ and a property is said to be true almost everywhere in $(m,T)$-capacity ($cap_{m,T}$-a.e) if it is satisfied outside a $ (m, T)$-pluripolar set.
  \begin{pro} The $(m,T)$-capacity admits the following properties:
\begin{enumerate}
\item If $E$ is a Borel set of $ \Omega $, then

$ cap_{m,T}(E,\Omega)= \sup \{\int_{E}(dd^{c}u)^{m-p}\wedge \beta^{n-m}\wedge T , u\in {\cal {SH}}_{m}(\Omega)\ and    -1\leqslant u \leqslant 0 \}.$
\item If $ E_{1}\subset E_{2}\subset\Omega $  then  $ cap_{m,T}(E_{1},\Omega)\leqslant cap_{m,T}(E_{2},\Omega).$
\item If $ E \subset \Omega \subset \Omega^{'} $ then $  cap_{m,T}(E,\Omega)\geqslant cap_{m,T}(E,\Omega^{'}) .$
 \item If $ E_{1},E_{2},...  $ are Borel subsets of $ \Omega $ then
$ cap_{m,T}(\cup_{j=1}^{\infty}E_{j},\Omega) \leqslant  \sum\limits_{j=1}^{\infty}cap_{m,T}(E_{j},\Omega).$
\item If $ E_{1}\subset E_{2}\subset...\subset\Omega $ are Borel sets then $cap_{m,T}(\cup_{j=1}^{\infty}E_{j},\Omega)=\lim \limits_{j\rightarrow+\infty}cap_{m,T}(E_{j},\Omega).$
\end{enumerate}
\end{pro}
 Denote by ${\cal {SH}}_{m}^{-}(\Omega)$ the set of negative $m$-sh functions and assume that $\Omega$ is an $m$-hyperconvex domain, that is
  it is bounded, connected and there exist an $m$-sh  function continuous and exhaustive on $\Omega $. In other words, there is a function $\varphi:\Omega\rightarrow[-\infty,0[$, $m$-sh  and continuous such that $\{z\in\Omega;\ \varphi(z)<c\}\Subset\Omega$ for every $ c<0 $. Following \cite{ref13}, we define the Cegrell classes:
\vskip0.1cm
$ \mathscr{E}^{m}_0(\Omega)=\{u\in {\cal {SH}}_m^-(\Omega)\cap L^\infty(\Omega), {\ds\lim_{z\rightarrow\partial\Omega}} u(z) = 0,\int_{\Omega } (dd^cu)^{m}\wedge\beta^{n-m} < +\infty\}$
\vskip0.1cm
$\mathscr{F}^{m}(\Omega)= \{u \in{\cal {SH}}_m^-(\Omega), \exists (u_{j})_j\subset {\mathscr E}^{m}_0(\Omega),u_{j}\downarrow u,\sup_{j}\int_{\Omega} (dd^c u_{j})^{m}\wedge\beta^{n-m}< +\infty\},$\\
and we define ${\mathscr E}^{m}(\Omega)$ by the set of $u\in {\cal {SH}}_m^-(\Omega)$ such that for every $z_{0} \in \Omega $, there exist a neighborhood $\omega$ of $z_{0}$ and a sequence $(u_{j})_{j} \subset {\mathscr E}^{m}_0 (\Omega)$ such that $u_{j}\downarrow u$ on $\omega$ and $\sup_{j} \int_{\Omega} (dd^c u_{j})^{m}\wedge\beta^{n-m} < \infty.$
In our work we need the following properties:
\begin{enumerate}
\item If $\varphi \in  {\mathscr K}\in\{{\mathscr E}_{0}^{m}(\Omega),{\mathcal F}^{m}(\Omega),{\mathscr E}^{m}(\Omega)\}$ and $\psi \in {\cal {SH}}_m^-(\Omega)$, then $ \max(\varphi,\psi) \in {\mathscr K}.$
\item If $ u\in \mathscr{E}^{m}(\Omega) $ then $u$ is locally in $\mathscr{F}^{m}(\Omega) $. That is to say, $ \forall u \in \mathscr{E}^{m}(\Omega)\ \ and\ \forall\ K\Subset \Omega$, it exists $v\in \mathscr{F}^{m}(\Omega) $ as $ u=v\ on\ K. $
\item If $ u \in {\mathscr F}^{m}(\Omega)$ then $\int_{\Omega} (dd^c u)^{m}\wedge\beta^{n-m}< +\infty$.
\end{enumerate}
\section{Complex Hessian operator and convergence in capacity}
\subsection{Complex Hessian operator}
According to Lu \cite{ref13}, the complex Hessian operator is well defined and continuous for decreasing sequences of locally bounded  $m$-sh functions. This result was generalized by Dhouib-Elkhadhra \cite{ref9}, by assuming that the $m$-sh functions are bounded near the boundary. For $ v\in{\cal {SH}}_{m}(\Omega)$, we denote by  ${\mathscr L}(v)$  the set of points $ z \in \Omega $ such that $ v $ is not bounded in any neighborhood of $ z $ and by $ {\cal H}_{s}$ the Hausdorff measure of dimension $ s $. When $T=1$, by using a technics go back to Demailly \cite{ref10} on the study of the Monge-Amp\`ere operator, Wan and Wang \cite{ref19} have generalized Proposition 2 for general $m$-sh functions with condition on the Hausdorff measure of the their common locus. Throughout the paper, we denote by  $\sigma_{T}=T\wedge \beta^{n-p}$ the trace measure of a current $T\in{\mathscr C}_p^m(\Omega)$. Then, relying on a recently work of \cite{ref1}, we prove analogically:
\begin{thm}
Let $ T\in{\mathscr C}_p^m(\Omega)$, $m\geqslant p+1$ and $(v_{j})_{j}$ a sequence of $m$-sh functions decreasing to an $m$-sh function $v\in L^1_{loc}(\sigma_T)$.  For every $ l=1,...,k\leqslant m-p, $ let  $(u_{l}^{j})_{j}$  be a sequence of $m$-sh functions decreasing to $u_{l}\in {\cal {SH}}_{m}(\Omega)\cap L^{\infty}_{loc}(\Omega)$. Then, $v_{j}dd^{c}u_{1}^{j}\wedge...\wedge dd^{c}u_{k}^{j}\wedge \beta^{n-m}\wedge T$ converges weakly to $vdd^{c}u_{1}\wedge...\wedge dd^{c}u_{k}\wedge\beta^{n-m} \wedge T.$ \end{thm}
When $m=n$, we recover a result of \cite{ref1} for the trivial current  $T=1$.
\begin{proof}
Thanks to Theorem 2.6 in \cite{lu article}, we have $$ S_{j}=dd^{c}u_{1}^{j}\wedge...\wedge dd^{c}u_{k}^{j}\wedge\beta^{n-m}\wedge T \ {\rm converges\  weakly \  to} \ S=dd^{c}u_{1}\wedge...\wedge dd^{c}u_{k}\wedge\beta^{n-m}\wedge T.$$ Based on the inequality 1.2.8 in \cite{ref13}, we prove that the sequence $(v_{j} S_{j})_{j} $ is locally uniformly bounded in masses. Hence, it  suffices to show that if  $ (v_{j}S_{j})_{j}$ converges weakly to $\Theta$ when $j\rightarrow +\infty$ then $ \Theta=vS.$ For this aim, since $S_j,S$ are positive currents, then for any strongly positive form $\varphi$ of bidegree $(m-p-k,m-p-k)$, one has $$v_{j}S_{j}\wedge\varphi\leqslant v_{j_{0}} S_{j}\wedge\varphi\leqslant v_{j_{0}}\star\rho_{\varepsilon} \ S_{j}\wedge\varphi,\quad \ \forall \varepsilon >0,\ \forall j\geqslant j_{0},$$ where $ v_{j_{0}}\star\rho_{\varepsilon}$ is a regularization by convolution of $ v_{j_{0}}.$ If $ j\rightarrow +\infty, $ we get
$\Theta\wedge\varphi\leqslant v_{j_{0}}\star\rho_{\varepsilon}S\wedge\varphi,$ then we obtain $\Theta\wedge\varphi \leqslant vS\wedge\varphi$ when $\varepsilon\rightarrow 0$ and $j_{0}\rightarrow +\infty$. This means that $\Theta\leqslant vS$.
Conversely, without loss of generality, we assume that all the functions are defined on the closed euclidean ball $\overline{{\mathscr B}_{z_0}(r)}=\overline{{\mathscr B}(z_{0},r)}$ and
for every $l=1,...,k$ we have $ u_{l}^{j}=u_{l}=A(|z-z_{0}|^{2}-r^{2}) $ in a neighborhood of $ \partial{\mathscr B}_{z_0}(r)\cap\Supp T $ with $A>0.$ Since $ \Theta\leqslant vS, $ it remains to prove $\int_{{\mathscr B}_{z_0}(r)} \left(vS-\Theta\right)\wedge\beta^{m-p-k}\leqslant 0.$ For this, an adaptation of the proof of Lemma 2.2 in \cite{ref1}, yields the following  corresponding result:
 \begin{lem}
Let $\Omega$ be a bounded domain of $\cb^{n},$ $ T\in{\mathscr C}_p^m(\Omega)$ and $  v,u_{0},u_{1},...,u_{m-p}$ are $m$-sh functions defined in a neighborhood of $\overline{ \Omega} .$  Assume that $u_{0},u_{1},...,u_{m-p}$ are locally bounded, $ u_{0}\leqslant u_{1} $ on $\Omega$ and $ u_{0}=u_{1} $ on $\Omega\cap{\mathscr O}$ with ${\mathscr O}$ is a neighborhood of $ \partial\Omega\cap \Supp T. $ Then $$ \ds\int_{\Omega}vdd^{c}u_{0}\wedge dd^{c}u_{2}\wedge...\wedge dd^{c}u_{m-p}\wedge \beta^{n-m}\wedge T\leqslant \int_{\Omega}vdd^{c}u_{1}\wedge dd^{c}u_{2}\wedge...\wedge dd^{c}u_{m-p}\wedge\beta^{n-m}\wedge T.$$
\end{lem}
Applying Lemma 2 $k$-times for $\Omega={\mathscr B}_{z_0}(r)$ and ${\mathscr O}$ is a neighborhood of $ \partial{\mathscr B}_{z_0}(r)\cap \Supp T$, we get $ \int_{{\mathscr B}_{z_0}(r)} v S\wedge\beta^{m-p-k}\leqslant \int_{{\mathscr B}_{z_0}(r)}  v_{j}S_{j}\wedge\beta^{m-p-k} .$ Let $K$ be a compact set such that ${\mathscr B}_{z_0}(r)\smallsetminus K\subset{\mathscr B}_{z_0}(r)\cap {\mathscr O}$. Then
  $$\begin{array}{lcl}
  \ds\int_{{\mathscr B}_{z_0}(r)}\left(vS-\Theta\right)\wedge\beta^{m-p-k}&\leqslant&\liminf\limits_{j\rightarrow +\infty}\ds\int_{{\mathscr B}_{z_0}(r)}\left(v_jS_j-\Theta\right)\wedge\beta^{m-p-k}\\&=&\liminf\limits_{j\rightarrow +\infty}\ds\int_{K}\left(v_jS_j-\Theta\right)\wedge\beta^{m-p-k}+\liminf\limits_{j\rightarrow \rightarrow+\infty}\ds\int_{{\mathscr B}_{z_0}(r)\smallsetminus K}\left(v_jS_j-\Theta\right)\wedge\beta^{m-p-k} \\&=&\liminf\limits_{j\rightarrow +\infty}\ds\int_{{{\mathscr B}_{z_0}(r)}\smallsetminus K}\left(A^kv_j\beta^{n-m+k}\wedge T-\Theta\right)\wedge\beta^{m-p-k}.
  \end{array}$$
  We have used the fact that the first term in the first equality equals to $0$. Moreover, since $v\in L^1_{loc}(\sigma_T)$ we have $A^kv_{j}\beta^{n-m+k}\wedge T $ converges weakly to $A^kv\beta^{n-m+k}\wedge T $ which equals to $\Theta$ on ${\mathscr B}_{z_0}(r)\smallsetminus K$. This imply that the second term is equals to $0$ also.
\end{proof}
Next, we shall direct our attention to the study of the convergence of the sequence of operators $( dd^{c}u_{1}^{j}\wedge...\wedge dd^{c}u_{q}^{j}\wedge \beta^{n-m}\wedge T_{j})_{j}$ where $T\in {\mathscr C}_{p}^{m}(\Omega)$, $T_{j}=T\star\chi_{j}$ is a smooth regularization by convolution of $T$ and $(u_{k}^{j})_{j}$ are  sequences of  $m$-sh functions not necessarily bounded and decreasing towards $u_{k}$ for all $1\leqslant k\leqslant q$. We strongly inspired from the technics of Ben Messaoud-El Mir \cite{me-el}, by investigating  firstly the convergence for the local potential associated to $T$. First we state
\begin{defn}
Let $u_{1},...,u_{q}$ are  $m$-sh functions  on $\Omega.$ We will say that $dd^{c}u_{1}\wedge...\wedge dd^{c}u_{q}\wedge \beta^{n-m} $
 is well defined if and only if we have the following induction on $ k = 2, ..., q $:\begin{enumerate}
 \item $u_k\in L^{1}_{loc}(dd^{c}u_{1}\wedge...\wedge dd^{c}u_{k-1}\wedge\beta^{n-m})$.
 \item For all $u_ {1}^{j},..., u_{k}^{j} $ decreasing sequences of $m$-sh functions  and which  converges locally to $u_ {1} ,. .., u_ {k}$ respectively, we have $$ u_{k}^{j}dd^{c}u_{1}^{j}\wedge...\wedge dd^{c}u^{j}_{k-1}\wedge \beta^{n-m}\L u_{k}dd^{c}u_{1}\wedge...\wedge dd^{c}u_{k-1}\wedge \beta^{n-m},\ weakly.$$ Setting $dd^{c}u_{1}\wedge ...\wedge dd^{c}u_{k}\wedge \beta^{n-m}=dd^{c}(u_{k}dd^{c}u_{1}\wedge...\wedge dd^{c}u_{k-1}\wedge \beta^{n-m})$ .
 \end{enumerate}
\end{defn}
\begin{exe} The current  $ dd^{c}u_{1}\wedge...\wedge dd^{c}u_{q}\wedge \beta^{n-m} $ is well defined  as soon as one of the following two conditions was satisfied: \begin{enumerate}
\item $\Omega$ is a bounded strictly pseudoconvex open subset of $\cb^{n}$  such that  ${\mathscr L}(u_{k})\cap\partial \Omega=\emptyset$ for all $1\leqslant k\leqslant q$ (see \cite{ref9}).
\item ${\cal H}_{2m-2l+1}({\mathscr L}(u_{i_{1}})\cap...\cap {\mathscr L}(u_{i_{l}}))=0$  for all
 $1\leqslant i_{1}<...<i_{l}\leqslant q$ (see \cite{ref19}).
\end{enumerate}
\end{exe}
Analogously to the complex setting, in order to study the convergence of the complex Hessian operator relatively to $T_j$, we shall pass by the local potential associated to $T$. Assume that $T$ is a current of order $0$ (i.e $T$  continues into the space of continuous functions) and of bidegree $(p,p)$ in $\Omega$ and let $\Omega_0\Subset\Omega$. Let $\eta\in{\mathscr D}(\Omega)$ such that $0\leqslant \eta\leqslant 1 $ and $\eta\equiv 1 $ in a neighborhood of $\overline\Omega_0$. Then the local potential $U=U(\Omega_0,T)$ associated to $T$ is the current of bidegree $(p-1,p-1)$ defined by:
 $$U(z) =\! -c_n\int_{x \in \cb^n}\!\eta (x) T(x)\wedge
{(dd^c|z-x|^2)^{n-1}\over|z-x|^{2n-2}},\qquad{\rm where} \quad c_n= \ds {1\over (n-1)(4\pi )^n }.$$
Denote by $p_1$ (resp.$p_2$) the first (resp.second) projection of $\cb^n\times\cb^n$ on $\cb^n$, i.e, $p_1(x,z)=x$ and $p_2(x,z)=z$. Let also $\tau$ be the function defined by $\tau(x,z)=z-x$. Denote by $K(t)=-c_n|t|^{2-2n}(dd^c|t|^2)^{n-1}$, then it is clear that $U=p_{2\star}[p_1^\star(\eta T) \wedge \tau^\star K]$ and $U_j=U\ast \chi_j=p_{2\star}[p_1^\star(\eta T) \wedge \tau^\star K_j]$, where $K_j=K\star\chi_j$. By an arguments of pull-back and direct image of forms and currents, we see that $U$ is (strongly) negative provided that $T$ is (strongly) positive. In the context of $m$-positivity, if $T\in \mathscr {C}_{p}^{m}(\Omega)$ then $T_j$ is also $m$-positive. In fact, for all $\alpha_{1},...,\alpha_{m-p}$ $(1,1)$ $m$-positives forms with constant coefficients (see the comment after Definition 1) and for a positive function $\varphi\in \mathscr{D}(\Omega)$, it is easy to see that $$\langle T_j\wedge\alpha_1\wedge...\wedge\alpha_{m-p}\wedge\beta^{n-m},\varphi\rangle=\ds\int_{x\in\cb^{n}}\ds\int_{y\in\cb^{n}}\chi_{j}(x-y)\varphi(x)\wedge T(y)\wedge \alpha_1\wedge...\wedge\alpha_{m-p}\wedge\beta^{n-m}dydx$$
As $\alpha_1,...,\alpha_{m-p}$ are $m$-positive $(1,1)$-forms with constant coefficients with respect to $y$ and $T$ is $m$-positive, it implies that the right hand side in the above equality is positive.
Unfortunately, $m$-positivity and $m$-subharmonicity are not preserved by holomorphic maps. Therefore, we restrict ourselves to a smaller subclass of ${\mathscr C}_p^m(\Omega)$. Recall that a current $T\in{\mathscr C}_p^m(\Omega)$, $m\geqslant p+1$ is said to belong to ${\mathscr M}_p^m(\Omega)$ if $T$ is of order zero and the associated local potential is $m$-negative and the sequence $(U_j)_j$ is decreases. Obviously, the class ${\mathscr M}_p^m(\Omega)$ contains at least the class of strongly positive currents. Now, we state our first main generalization to the complex hessian theory of a result due to Ben Messaoud-El Mir \cite{me-el} for the complex setting.
\vskip0.15cm
\noindent {\bf Theorem A.} {\it Assume that $T\in{\mathscr M}_p^m(\Omega)$ and $U=U(\Omega_0,T)$ is the associated local potential to $T$, where $\Omega_{0}\Subset \Omega$ is strictly pseudoconvex. Let $ u_{1}, ..., u_{q}\in {\cal {SH}}_{m}(\overline{\Omega}_{0})$ satisfying ${\mathscr L}(u_{k})\cap \partial\Omega_{0}\cap \Supp (T)=\emptyset,$ for $ k=1,...,q\leqslant m-p+1$ such that  $dd^{c}u_{1}\wedge...\wedge dd^{c}u_{q}\wedge \beta^{n-m}$ is well defined. Let $u_{1}^{j},...,u_{q}^{j}$ are sequences of $m$-sh functions decreasing respectively to $u_1,...,u_{q}$. Then,
\begin{enumerate}
\item $  dd^{c}u_{1}\wedge...\wedge dd^{c}u_{q}\wedge \beta^{n-m}\wedge U_{j}\L dd^{c}u_{1}\wedge...\wedge dd^{c}u_{q}\wedge\beta^{n-m}\wedge U$ weakly on $\Omega_{0}$.
\item If $q< m-p+1$, then
\begin{enumerate}
\item $ dd^{c}u_{1}^{j}\wedge...\wedge dd^{c}u_{q}^{j}\wedge\beta^{n-m}\wedge U \L dd^{c}u_{1}\wedge...\wedge dd^{c}u_{q}\wedge\beta^{n-m}\wedge U$ weakly on $\Omega_{0}$.
\item $ dd^{c}u_{1}^{j}\wedge...\wedge dd^{c}u_{q}^{j}\wedge\beta^{n-m}\wedge U_{j}\L dd^{c}u_{1}\wedge...\wedge dd^{c}u_{q}\wedge\beta^{n-m}\wedge  U$ weakly on $\Omega_{0}$.
\end{enumerate}
\end{enumerate}}
\begin{proof}  By monotonicity of $(U_j)_j$, the statement (b) is an immediate consequence of $(1)$ and (a). We argue as in \cite{me-el}, so we can assume that $ \Omega_{0} =\{z \in \cb ^ {n}; \rho (z) <0 \},$ where $\rho$ is a smooth strictly psh function in the neighborhood of $ \overline {\Omega}_{0}.$ According to Lemma 2.7 in \cite{me-el}, we have $dd^c U_j=\eta T_j+R_j$, where $R_j$ is a $(p,p)$-form with uniformly bounded coefficients on $ \Omega_{0} $. Hence, there exists a constant $A> 0$ such that $R_j+A (dd^{c}\rho)^p$ is strongly positive and therefore is $m$-positive. Since $T_j$ is $m$-positive, it follows that $ dd^{c}(U _{j} + A \rho (dd^{c}\rho)^ {p-1})  $ is $m$-positive. Moreover, since $U$ is $m$-negative then also for $U_j$ and therefore $S_{j}=U_{j}+A\rho(dd^{c}\rho)^{p-1} $ and $S=U+A\rho(dd^{c}\rho)^{p-1}$ are $m$-negative.
On the other hand, in view of Theorem 2 in \cite{ref9}, it suffices to prove the statements for $S_{j}$ and $S$ instead  of $U _{j}$ and $U$.
As $\Supp (T)\cap {\mathscr L}(u_{k})\Subset \Omega_{0}$, consider $ \delta>0 $ small enough so that $\cup_{k=1}^{q} {\mathscr L}(u_{k})\cap\Supp (T)\subset \Omega_{\delta}=\{\rho<-\delta\}$.
Let $g\in{\mathscr D} (\Omega_{0})$ positive such that $ g\equiv 1 $ on $ \Omega_{\delta} $. Since the closed subsets $ \Supp (dg) \cap \Supp(T) $ and $ \cup_{k} {\mathscr L}(u_{k}) $ are disjoint, consider a function $ f \in {\mathscr C}^{\infty} (\cb^{n}) $ positive such that $ f\equiv 1 $ in the neighborhood of the first subset and $f\equiv 0 $ in the neighborhood of the second. (1) It is enough to show that the sequence $$ B_{j} = \int_{\Omega_{0}} g dd^{c} u_ {1} \wedge ... \wedge dd^{c} u_{q} \wedge \beta^{n-m}\wedge S_{j}\wedge (dd ^{c} \rho )^{m-p-q+1} $$ is convergent. For a real constant $b$, we  have: $$\begin{array}{lcl} B_{j}&=&\ds\int_{\Omega_{0}}(\rho+b) dd^{c}u_{1}\wedge...\wedge dd^{c}u_{q} \wedge \beta^{n-m}\wedge dd^{c}(gS_{j})\wedge (dd ^{c} \rho )^{m-p-q}\\&=&\ds\int_{\Omega_{0}} (\rho+b) dd^{c}u_{1}\wedge ...\wedge dd^{c}u_{q}\wedge \beta^{n-m}\wedge g dd^{c}S_{j}\wedge(dd ^{c} \rho )^{m-p-q}\\&+& 2\ds\int_{\Omega_{0}}(\rho+b) dd^{c}u_{1}\wedge ...\wedge dd^{c}u_{q}\wedge \beta^{n-m} \wedge dg\wedge d^{c}S_{j}\wedge (dd ^{c} \rho )^{m-p-q}\\&+&\ds\int_{\Omega_{0}} (\rho+b) dd^{c}u_{1}\wedge ...\wedge dd^{c}u_{q}\wedge \beta^{n-m}\wedge dd^{c}g\wedge S_{j}\wedge(dd ^{c} \rho )^{m-p-q}\\&=&B_j^1+B_j^2(dg)+B_j^3.\end{array}$$
Let's choose $b$ so that $ \rho + b> 0 $ in the neighborhood of $ \overline {\Omega}_{0} $. Since  $dd ^{c} S_{j} $  is $m$-positive, it is clear that $B_j^1\geqslant 0$, and therefore we have $0\geqslant B_j\geqslant B_j^2(dg)+B_j^3$. Since the sequence $(B_j)_j$ is decreases, it suffices to show the convergence of $B_j^2(dg)$ and $B_j^3$. We shall write $ B_{j}^{2} (dg)= B_{j} ^ {2} (fdg) + B_ {j}^{2}((1-f) dg)$. The coefficient of  the form $fdg$ belongs to ${\mathscr D}(\Omega_{0})$ and vanishing in the neighborhood of $ \cup_{k} {\mathscr L}(u_{k})$. Therefore the  functions $u_{k}$ are locally bounded in the neighborhood of $\Supp(fdg)$. Then in view of Theorem  7 in \cite{ref9}, we have $$\begin{array}{lcl}&\lim\limits_{j\rightarrow +\infty}&\ds\int_{\Omega_{0}}(\rho+b) dd^{c}u_{1}\wedge ...\wedge dd^{c}u_{q}\wedge \beta^{n-m}\wedge fdg \wedge d^{c}S_{j}\wedge (dd ^{c} \rho )^{m-p-q}\\&=&\ds\int_{\Omega_{0}
 }(\rho+b) dd^{c}u_{1}\wedge ...\wedge dd^{c}u_{q}\wedge \beta^{n-m}\wedge fdg \wedge d^{c}S\wedge (dd ^{c} \rho )^{m-p-q}\end{array}$$ On the other hand, the form $ (1-f) dg $ vanishes in a neighborhood ${\mathscr O} $ of $ \Supp (T) $. Since the singular support of $U$ is included in the support of $T$ then $ S_{j} $ converges to $S$ in $ {\mathscr C} ^ {\infty}_{p-1,p- 1} (\Omega_{0} \smallsetminus \overline {\mathscr O}) $
and $(1-f)dg \wedge d^{c} S_{j} $ converges to $ (1-f)dg \wedge d^{c}S $ in ${\mathscr D} _ {p,p} (\Omega_{0} \smallsetminus \overline {\mathscr O}) $.
By hypotheses, the current $dd^{c}u_{1}\wedge...\wedge dd^{c}u_{q}\wedge \beta^{n-m}$ is well defined in $\Omega_{0}$ then we have: $$\lim_{j\rightarrow+\infty}B_{j}^{2}((1-f)dg)=\ds \int_{\Omega_{0}}(\rho+b) dd^{c}u_{1}\wedge...\wedge dd^{c}u_{q}\wedge \beta^{n-m}\wedge (1-f)dg\wedge d^{c}S\wedge (dd ^{c} \rho )^{m-p-q}$$ Consequently, the convergence of $B_{j}^{2}$ follows. A similar arguments, give:
$$\lim\limits_{j\rightarrow +\infty}B_j^3= \ds\int_{\Omega_{0}}(\rho+b) dd^{c}u_{1}\wedge...\wedge dd^{c}u_{m-p}\wedge \beta^{n-m}\wedge dd^{c}g\wedge S$$
(2) (a) Let $ V $ be a neighborhood of $ (\Omega_ {0}\smallsetminus\Omega_{\delta})\cap\Supp (T)$ such that $\overline{V} \cap {\mathscr L}(u_{k}) = \emptyset, \ \forall k \in \{1, ...,q\}.$ In particular, the functions $u_{k}$  are locally bounded in $V $. Without loss of generality, we suppose that there exists a constant $M>1$ such that for all $k = 1, ..., q $, we have: $ -M \leqslant u_{k} \leqslant -1 $ on $ V.$  Let $ V^{'} \Subset V $ satisfying the same property as $V$ and $ \varphi $ be a  smooth function with compact support in a neighborhood of $ \overline {\Omega}_{0} $ such that $\varphi \equiv 1 $ on $ \overline {V ^{'}} \cap \overline {\Omega}_{0} $ and $ \varphi \equiv -2M $ on $ \overline {\Omega}_{0} \smallsetminus V$. Let $ A>\frac {2} {\delta} $ be a constant large enough  so that $ \varphi + A (\rho + \delta)$ is a psh function on $\overline{\Omega}_{0}$. Setting
 $$v_{k}(z)= \left \{
   \begin{array}{l c l}
        u_{k}(z)\ \ \quad \qquad\qquad\qquad \qquad z\in \Omega_{0}\smallsetminus V\\
      \max(\varphi+A\rho(z),u_{k}(z))\qquad z\in V
   \end{array}
   \right .$$
 On  $\overline{\Omega}_{0}\cap \partial V$, we have $\varphi+A\rho< -M$ then $v_{k}=u_{k}$ in a neighborhood of  $\overline{\Omega}_{0}\cap \partial V$ which implies that $v_{k}$ is an $m$-sh function on $\Omega_{0}$  and the current  $dd^{c}v_{1}\wedge...\wedge dd^{c}v_{q}\wedge\beta^{n-m}$ is well defined.
Moreover, we have $v_{k}=\varphi+A\rho$ on $V^{'}\cap\{\rho> -\frac{2}{A}\}$. By the same procedure of the max, we define the functions $ v_{k}^{j}$ relative to $u_ {k}^{j}$. Let $ \Theta $ be a weak limit of $ (  dd^{c} v_{1}^{j} \wedge ... \wedge dd^{c} v_{q}^{ j} \wedge \beta^{n-m}\wedge
 S)_{j} $. By the monotonicity of $ (S_{j})_{j} $ and in view of $(1)$, it is easy to see that $\Theta\leqslant  dd^{c}v_{1}\wedge...\wedge dd^{c}v_{q}\wedge \beta^{n-m}\wedge S$.
Let $ r \in\nb^{\ast}$, since $dd^cS_r$ is $m$-positive, the Stokes formula implies that for all positive function $ g \in \mathscr {D} (\Omega_ {0}) $, we have:
$$\begin{array}{lcl}
   &&\ds\int_{\Omega_{0}} dd^{c}v_{1}^{j}\wedge...\wedge dd^{c}(v_{q}^{j}-v_{q})\wedge \beta^{n-m}\wedge gS_{r}\wedge (dd^{c}\rho)^{m-p-q+1}\\ &\geqslant& 2\ds\int_{\Omega_{0}} (v_{q}^{j}-v_{q}) dd^{c}v_{1}^{j}\wedge...\wedge dd^{c}v_{q-1}^{j}\wedge \beta^{n-m}\wedge dg\wedge d^{c}S_{r}\wedge (dd^{c}\rho)^{m-p-q+1}\\&+&\ds\int_{\Omega_{0}}(v_{q}^{j}-v_{q})  dd^{c}v_{1}^{j}\wedge...\wedge dd^{c}v^{j}_{q-1}\wedge \beta^{n-m}\wedge dd^{c}g
\wedge S_{r}\wedge (dd^{c}\rho)^{m-p-q+1} .
\end{array}$$ Choose $ g $ such that $ g \equiv 1 $ on $ \overline {\Omega} _ {\frac {2} {A}} $ and  let $ f\in {\mathscr D} (V ^ {'}) $ such as $ f \equiv 1 $ on $ \Supp (dg) \cap \Supp (T) $. As $\Supp (fdg)\subset V^{'}\cap \{\rho> -\frac{2}{A}\}$, where $v_{k}^{j}=v_{k}=\varphi+A\rho$ then  $(v_{k}^{j}-v_{k})dg=(v_{k}^{j}-v_{k})(1-f)dg.$ Furthermore, $ (1-f) dg $ vanishes in the neighborhood $\mathscr O$ of $ \Supp (T) $, $S$ is of class $ {\mathscr C}^{\infty}$ in the neighborhood of $ \Supp ((1-f)dg) $ hence the sequence $ (1-f)dg \wedge d^{c} S_{r} $ converges into $ {\mathscr D} _{p, p} (\Omega_{0}) $ to $(1-f) dg \wedge d^{c}S $. By using the above inequality  and when $ r \rightarrow +\infty $, we get: $$\begin{array}{lcl}
   &&\ds\int_{\Omega_{0}} dd^{c}v_{1}^{j}\wedge...\wedge dd^{c}(v_{q}^{j}-v_{q})\wedge \beta^{n-m}\wedge gS\wedge (dd^{c}\rho)^{m-p-q+1}\\ &\geqslant& 2\ds\int_{\Omega_{0}} (v_{q}^{j}-v_{q}) dd^{c}v_{1}^{j}\wedge...\wedge dd^{c}v_{q-1}^{j}\wedge \beta^{n-m}\wedge (1-f)dg\wedge d^{c}S\wedge (dd^{c}\rho)^{m-p-q+1}\\&+&\ds\int_{\Omega_{0}}(v_{q}^{j}-v_{q})  dd^{c}v_{1}^{j}\wedge...\wedge dd^{c}v_{q-1}^{j}\wedge \beta^{n-m}\wedge (1-f)dd^{c}g
\wedge S\wedge (dd^{c}\rho
)^{m-p-q+1}.\end{array}$$
 For each  $k=1,...,q-1$, setting $$R_{k}^{j}=dd^{c}v_{1}^{j}\wedge ...\wedge dd^{c}v_{k-1}^{j}\wedge  dd^{c}v_{k+1}\wedge...\wedge dd^{c}v_{q}\wedge \beta^{n-m}\wedge (dd^{c}\rho)^{m-p-q+1}$$ and $D_{k}^{j}=2\ds\int_{\Omega_{0}}(v_{k}^{j}-v_{k})(1-f)dg\wedge d^{c}S\wedge R_{k}^{j}+\ds\int_{\Omega_{0}}(v_{k}^{j}-v_{k})(1-f)dd^{c}g\wedge S\wedge R_{k}^{j}$. Since $dd^{c}v_{1}\wedge...\wedge dd^{c}v_{q}\wedge \beta^{n-m}$ is well defined, it is clear that $\lim\limits_{j\rightarrow +\infty}D_k^j=0$. By iterating the above inequality, we obtain: $$\begin{array}{lcl}&  \ds\int_{\Omega_{0}}& dd^{c}v_{1}^{j}\wedge...\wedge dd^{c}v_{q}^{j}\wedge \beta^{n-m}\wedge gS\wedge (dd^{c}\rho)^{m-p-q+1}\\&\geqslant&\ds\int_{\Omega_{0}}dd^{c}v_{1}\wedge...\wedge dd^{c}v_{q}\wedge \beta^{n-m}\wedge gS\wedge (dd^{c}\rho)^{m-p-q+1} +\sum_{k=1}^{q-1}D_k^j .\end{array}$$
Consequently, when  $ j \rightarrow +\infty $, we get $$\ds\int_{\Omega_{0}}g\Theta\wedge (dd^{c}\rho)^{m-p-q+1}\geqslant \ds\int_{\Omega_{0}} dd^{c}v_{1}\wedge...\wedge dd^{c}v_{q}\wedge \beta^{n-m}\wedge gS\wedge (dd^{c}\rho)^{m-p-q+1} $$ As the characteristic function $ \1_{\Omega_{0}} $ is the limit of an increasing sequence of such function $ g $, it is not difficult to deduce that $$\int_{\Omega_{0}}\Theta\wedge (dd^{c}\rho)^{m-p-q+1}\geqslant \int_{\Omega_{0}} dd^{c}v_{1}\wedge...\wedge dd^{c}v_{q}\wedge \beta^{n-m}\wedge S\wedge  (dd^{c}\rho)^{m-p-q+1}.$$ It follows that the positive current $E= dd^{c}v_{1}\wedge...\wedge dd^{c}v_{q}\wedge \beta^{n-m}\wedge S-\Theta$ of bidimension $(m-p-q+1,m-p-q+1)$,  satisfies $dd^{c} E =0$. Also for all points in a neighborhood of  $ \Omega_{0}\smallsetminus \overline {\Omega}_{\delta}$,  the functions $ v_{k}$ are bounded  and we have $$dd^{c}v_{1}\wedge...\wedge dd^{c}v_{q}\wedge\beta^{n-m}\wedge S=dd^{c}v_{1}^{j}\wedge...\wedge dd^{c}v_{q}^{j}\wedge\beta^{n-m}\wedge S=(dd^{c}(\varphi+A\rho))^{q}\wedge \beta^{n-m}\wedge S,$$ which implies that $E=0$ near the point. Hence, $E$ has a compact support and therefore $E=0$ in $\Omega_{0}$. Then, $dd^{c}v_{1}^{j}\wedge...\wedge dd^{c}v_{q}^{j}\wedge \beta^{n-m}\wedge S $ converges weakly to $dd^{c}v_{1} \wedge...\wedge dd^{c}v_{q} \wedge\beta^{n-m}\wedge S$. Since in $\Omega_{\delta}$, $v_{k}^{j}=u_{k}^{j}$  and $v_{k}=u_{k}$ then $dd^{c}u_{1}^{j}\wedge...\wedge dd^{c}u_{q}^{j}\wedge\beta^{n-m}\wedge S$ converges weakly to $dd^{c}u_{1}\wedge...\wedge dd^{c}u_{q}\wedge\beta^{n-m}\wedge S$. Therefore by using  Theorem 2 in \cite{ref9}, we deduce the convergence of $ dd^{c}u_{1}^{j}\wedge...\wedge dd^{c}u_{q}^{j}\wedge \beta^{n-m}\wedge U$ to $ dd^{c}u_{1}\wedge...\wedge dd^{c}u_{q}\wedge \beta^{n-m}\wedge U$.\end{proof}
A family of $m$-sh functions on $\Omega$, $u_{1},...,u_{q}$, are said satisfying the condition ${\mathscr C}_T$, if $\Omega$ is covered by $(\Omega_s)_s$, where $\Omega_s$ is strictly pseudoconvex, $\Omega_s\Subset\Omega$ and ${\mathscr L}(u_{k})\cap\partial\Omega_s\cap \Supp T=\emptyset$ for all $s,k$. Now, based on Theorem A, we are prepared to prove our main result in this section which is the counterpart in the complex hessian context of the main result of Ben Messaoud-El Mir \cite{me-el}.
\vskip0.15cm
\noindent{\bf Theorem B.} {\it Assume that $T\in {\mathscr M}_{p}^{m}(\Omega)$ and $u_{0},u_{1},...,u_{q}$ are $m$-sh functions on $\Omega$, satisfying condition ${\mathscr C}_T$. If the current $dd^{c}u_{0}\wedge...\wedge dd^{c}u_{q}\wedge\beta^{n-m}$ is well defined and if $u_{0}^{j},...,u_{q}^{j}$ are decreasing sequences of $m$-sh functions such that $ u_{k}^{j} $ converges pointwise to $u_{k}$, then, we have:
\begin{enumerate}
\item[(1)] $ dd^{c}u_{1}^{j}\wedge...\wedge dd^{c}u_{q}^{j}\wedge\beta^{n-m}\wedge T_{j}\L   dd^{c}u_{1}\wedge...\wedge dd^{c}u_{q}\wedge\beta^{n-m}\wedge T$ weakly on $\Omega$.
\item[(2)]  For all  $1\leqslant q< m-p$, we have: $$ u_{0}^{j} dd^{c}u_{1}^{j}\wedge...\wedge dd^{c}u_{q}^{j}\wedge \beta^{n-m}\wedge T_{j}\L  u_{0} dd^{c}u_{1}\wedge...\wedge dd^{c}u_{q}\wedge\beta^{n-m}\wedge T\  weakly\ on\ \Omega.$$
\end{enumerate}}
\begin{rem} Note that the second statement of Theorem B is far form being true when $q=m-p$. Indeed, assume that $T=1$ and  $u_0=...=u_{m}=\v_m(z)=-|z|^{-2\left(\frac{n}{m}-1\right)}$. It is well-known that $\v_m$ is $m$-sh and a straightforward computation yields:
$(dd^c \v_m)^{m}\wedge\beta^{n-m}=C_{n,m}\delta_0.\beta^n,$
where $C_{n,m}$ is a constant depending on $n$ and $m$ and $\delta_0$ is the Dirac measure on $0$. In particular, it is clear that the current $\v_m(dd^c \v_m)^{m}\wedge\beta^{n-m}$ hasn't a finite locally masses near $0$.
\end{rem}
\begin{proof}
(1)  Let $ \Omega_ {0} \Subset \Omega $ be a  strictly pseudoconvex subset as in the statement of Theorem A. By the weak continuity of $ dd^{c} $, Lemma 2.7 in \cite{me-el} and Theorem A, we have the weak convergence
$$ dd^{c}u_{1}^{j}\wedge...\wedge dd^{c}u_{q}^{j}\wedge \beta^{n-m}\wedge T_{j}+ dd^{c}u_{1}^{j}\wedge...\wedge dd^{c}u_{q}^{j}\wedge \beta^{n-m}\wedge R_{j}\L dd^{c}u_{1}\wedge...\wedge dd^{c}u_{q}\wedge \beta^{n-m}\wedge dd^cU.$$
As $R$ is smooth on  ${\Omega}_{0}$ and the current $dd^{c}u_{1}\wedge...\wedge dd^{c}u_{q}\wedge \beta^{n-m} $ is well defined, the weak convergence $ dd^{c}u_{1}^{j}\wedge...\wedge dd^{c}u_{q}^{j}\wedge \beta^{n-m}\wedge R\L  dd^{c}u_{1}\wedge...\wedge dd^{c}u_{q}\wedge \beta^{n-m}\wedge R$ was also guaranteed. Since, $ dd^{c}u_{1}^{j}\wedge...\wedge dd^{c}u_{q}^{j}\wedge \beta^{n-m}\wedge T_{j}+ dd^{c}u_{1}^{j}\wedge...\wedge dd^{c}u_{q}^{j}\wedge \beta^{n-m}\wedge(R_{j}-R)$ converges weakly to $  dd^{c}u_{1}\wedge...\wedge dd^{c}u_{q}\wedge \beta^{n-m}\wedge T$ where $R_{j}-R\in {\mathscr C}^{\infty}(\overline{\Omega}_{0})$ and converges uniformly to $0$. Hence, we deduce the weak convergence $ dd^{c}u_{1}^{j}\wedge...\wedge dd^{c}u_{q}^{j}\wedge \beta^{n-m}\wedge T_{j}\L   dd^{c}u_{1}\wedge...\wedge dd^{c}u_{q}\wedge \beta^{n-m}\wedge T.$\\
(2) Consider the functions $v_{0},...,v_{q}$ and  $v_{0}^{j},...,v_{q}^{j}$ obtained  by the procedure of the max used in the proof of Theorem A. Let $\Theta$ be a weak limit of $(v_{0}^{j} dd^{c}v_{1}^{j}\wedge...\wedge dd^{c}v_{q}^{j}\wedge \beta^{n-m}\wedge T_{j})_{j}$. By regularizing $v_{0}$ and $v_{0}^{j}$, a simple computation gives $\Theta\leqslant v_{0} dd^{c}v_{1}\wedge...\wedge dd^{c}v_{q}\wedge \beta^{n-m}\wedge T$. The closed positive current $E=v_{0} dd^{c}v_{1}\wedge...\wedge dd^{c}v_{q}\wedge \beta^{n-m}\wedge T-\Theta$ is of bidimension $(m-p-q,m-p-q)$. Near $ \partial\Omega_{0} $, the functions $ v_{k} $ are bounded, then $ E $ has  a compact support in $ \Omega_ {0}$. Now according to the previous argument, we obtain that  $ dd^{c}E=0$, and therefore $E=0$.
\end{proof}
Now, we end this subsection by extending the previous results for the class $\delta\cal {SH}_m(\Omega)$. Recall that a function $u$ belongs to $\delta\cal {SH}_m(\Omega)$ if there exists  $u_{1},u_2\in \cal {SH}_m(\Omega)$ such that $u=u_1-u_2$ on $\Omega$. For $u\in \delta \cal {SH}_m(\Omega)$, we say that the complex Hessian operator $(dd^cu)^q\wedge\beta^{n-m}$ is well defined if for all $0\leqslant s\leqslant q\leqslant m$, $(dd^{c}u_{1})^{s}\wedge (dd^{c}u_{2})^{q-s}\wedge \beta^{n-m}$ is well defined (see Definition 2) and therefore, we put
$$(dd^cu)^q\wedge\beta^{n-m}=\sum\limits_{s=0}^q(-1)^{q-s}
\binom{q}{s}(dd^cu_1)^{s}\wedge (dd^{c}u_2)^{q-s}\wedge\beta^{n-m}.$$
 \begin{pro}
 Assume that  $T\in \mathscr{M}_{p}^{m}(\Omega)$, $\Omega_0\Subset\Omega$ is strictly pseudoconvex and $u=u_1-u_2\in \delta\cal{SH}_m(\overline{\Omega}_0)$ such that ${\mathscr L}(u_i)\cap\Supp T\cap \partial\Omega_{0}=\emptyset$ for $i=1,2$ and $(dd^{c}u)^q\wedge\beta^{n-m}$ is well defined. Let $(u_j=u_{j,1}-u_{j,2})_j\subset \delta\cal{SH}_m(\overline{\Omega}_0)$ where $u_{j,1}$, $u_{j,2}$ are decreasing sequence respectively to $u_1,u_2$. Then for all $q\leqslant m-p+1$, we have
 \begin{enumerate}
\item $(dd^{c}u)^q\wedge\beta^{n-m}\wedge U_j$ converges weakly on $\Omega_0$. We denote this limit by $(dd^{c}u)^q\wedge\beta^{n-m}\wedge U.$
\item If $q< m-p+1$, then
\begin{enumerate}
\item $(dd^{c}u_{j})^q\wedge\beta^{n-m}\wedge U$ converges weakly to $(dd^{c}u)^q\wedge\beta^{n-m}\wedge U$ on $\Omega_0.$
\item $(dd^{c}u_{j})^q\wedge\beta^{n-m}\wedge U_{j}$ converges weakly to $(dd^{c}u)^q\wedge\beta^{n-m}\wedge U$ on $\Omega_0.$
\end{enumerate}
\end{enumerate}
 \end{pro}
\begin{proof}
(1) Since $U_j$ is smooth, then
$$(dd^cu)^q\wedge\beta^{n-m}\wedge U_j=\sum\limits_{s=0}^q(-1)^{q-s}\binom{q}{s}(dd^cu_1)^{s}\wedge (dd^{c}u_2)^{q-s}\wedge\beta^{n-m}\wedge U_j.$$
Thanks to Theorem A, the sequence $(dd^cu_1)^{s}\wedge (dd^{c}u_2)^{q-s}\wedge\beta^{n-m}\wedge U_j $ converges weakly to $(dd^cu_1)^{s}\wedge (dd^{c}u_2)^{q-s}\wedge\beta^{n-m}\wedge U$, then the sequence $(dd^cu)^q\wedge\beta^{n-m}\wedge U_j$ converges to a limit denoted by $ (dd^cu)^q\wedge\beta^{n-m}\wedge U$ such that
 $$(dd^cu)^q\wedge\beta^{n-m}\wedge U=\sum\limits_{s=0}^q(-1)^{q-s}\binom{q}{s}(dd^cu_1)^{s}\wedge (dd^{c}u_2)^{q-s}\wedge\beta^{n-m}\wedge U.$$
(2) (a) In view of $(1)$, we have
$$(dd^cu_j)^q\wedge\beta^{n-m}\wedge U= \sum\limits_{s=0}^q (-1)^{q-s}\binom{q}{s} (dd^{c}u_{j,1})^s\wedge (dd^cu_{j,2})^{q-s}\wedge\beta^{n-m}\wedge U.$$
Once again, Theorem A imply that $(dd^{c}u_{j,1})^s\wedge (dd^cu_{j,2})^{q-s}\wedge\beta^{n-m}\wedge U$ converges weakly to $(dd^{c}u_{1})^s\wedge (dd^cu_{2})^{q-s}\wedge\beta^{n-m}\wedge U$. It follows that
$(dd^cu_j)^q\wedge\beta^{n-m}\wedge U$ converges weakly to $$\sum\limits_{s=0}^q (-1)^{q-s} \binom{q}{s}(dd^{c}u_{1})^s\wedge (dd^cu_{2})^{q-s}\wedge\beta^{n-m}\wedge U=(dd^cu)^q\wedge\beta^{n-m}\wedge U.$$
(b) By statement (b) in Theorem A and in view of (1), it is not difficult to deduce that the sequence $(dd^cu_j)^q\wedge\beta^{n-m}\wedge U_j$ converges weakly to $(dd^cu)^q\wedge\beta^{n-m}\wedge U$. \end{proof}
\ As an immediate consequence of Proposition 4 and by the same lines of proof of Theorem B, it is not hard to deduce the following version for the class $\delta\cal{SH}_m(\Omega)$.
\begin{thm}
Assume that  $T\in \mathscr{M}_{p}^{m}(\Omega)$, $\Omega_0\Subset\Omega$ is strictly pseudoconvex and $u=u_1-u_2\in \delta\cal{SH}_m(\overline{\Omega}_0)$ such that ${\mathscr L}(u_i)\cap\Supp T\cap \partial\Omega_{0}=\emptyset$ for $i=1,2$ and  $(dd^{c}u)^q\wedge\beta^{n-m}$ is well defined. Let $(u_j=u_{j,1}-u_{j,2})_j\subset \delta\cal{SH}_m(\overline{\Omega}_0)$ where $u_{j,1}$, $u_{j,2}$ are decreasing sequence respectively to $u_1,u_2$. Then for all $ 1\leqslant q\leqslant m-p$, we have \begin{enumerate}
\item $(dd^cu_{j})^q\wedge\beta^{n-m}\wedge T_j\rightarrow (dd^cu)^q\wedge\beta^{n-m}\wedge T $ weakly on $\Omega_0$.
\item If $ q< m-p$, $u_j(dd^cu_j)^q\wedge \beta^{n-m}\wedge T_j \rightarrow u(dd^cu)^q\wedge \beta^{n-m}\wedge T$ weakly on $\Omega_0$.
\end{enumerate}
\end{thm}
\subsection{Convergence in capacity}
Until the work of Xing \cite{ref21}, convergence in capacity becomes an effective tools in studying convergence of complex and complex Hessian operators. It is well known that a decreasing sequence of locally bounded $m$-sh functions $(u_j)_j$, which converges to  $u$, it converges to $ u$ in $ cap_{m,T} $  on every $E\Subset \Omega$. This means that for every $\delta>0,$ $ \lim\limits_{j\rightarrow +\infty}cap_{m,T}\left(E\cap\{|u_{j}-u|\geqslant\delta\}\right)=0$. For every Borel subset $E\subset\Omega,$ let's define

$cap_{m-1,T}(E)=\sup\{\int_{E}(dd^{c}v)^{m-p-1}\wedge\beta^{n-m+1}\wedge T, v\in \cal{SH}_m(\Omega),\   -1\leqslant v\leqslant 0\}. $
\vskip0.1cm
\noindent Assuming that $\Omega$ is bounded, it is not hard to see that $cap_{m-1,T}$ is dominated from above by $cap_{m,T}$. Therefore, convergence in $cap_{m,T}$-capacity leads to convergence in $cap_{m-1,T}$-capacity. In this way, we get the following result which generalizes the one given by Xing \cite{ref21} for the particular cases $m=n$ and $T=1$.
\begin{thm}
Let $T\in {\mathscr C}_{p}^{m}(\Omega)$ and $(u_{j})_{j}$ be a sequence of locally uniformly bounded $m$-sh functions. Suppose that $u\in \cal{SH}_{m}(\Omega)\cap L^{\infty}(\Omega)$ and $u_{j}\rightarrow u$ in $cap_{m-1,T}$ then for all $q\leqslant m-p,$ $$(dd^{c}u_{j})^{q}\wedge\beta^{n-m}\wedge T\L (dd^{c}u)^{q}\wedge\beta^{n-m}\wedge T\ weakly \ on \ \Omega.$$\end{thm}
\begin{proof}
We shall prove by induction that for each $k\leqslant m-p,$ $(dd^{c}u_{j})^{k}\wedge\beta^{n-m}\wedge T$ converges weakly to $(dd^{c}u)^{k}\wedge\beta^{n-m}\wedge T$. For $k=1,$  the convergence assumption implies that $u_{j}\rightarrow u$ in $L^{1}_{loc}(\sigma_{T})$. Indeed, for every $K\Subset \Omega, \delta>0$, since $\Omega$ is bounded there exists a constant $C>0$ not depending on $j$ such that $$\begin{array}{lcl}
\ds\int_{K}|u_{j}-u|\beta^{n-p}\wedge T&=&\ds\int_{K\cap\{|u_{j}-u|\geqslant \delta\}}|u_{j}-u|\beta^{m-p-1}\wedge\beta^{n-m+1}\wedge T+\ds\int_{K\cap\{|u_{j}-u|< \delta\}}|u_{j}-u|\beta^{n-p}\wedge T\\\\&\leqslant& C\|u_{j}-u \|_{L^{\infty}(K)}cap_{m-1,T}(K\cap\{|u_{j}-u|\geqslant \delta\})+\delta \sigma_{T}(K)
\end{array}$$ then we deduce the result by letting $j\l +\infty$ and by arbitrariness of $\delta$. Hence, it follows that $dd^{c}u_{j}\wedge\beta^{n-m}\wedge T$ converges weakly to $dd^{c}u\wedge\beta^{n-m}\wedge T$. Assume that Theorem 3 is true for all $k=q< m-p$ and we shall prove that $u_{j}(dd^{c}u_{j})^{q}\wedge \beta^{n-m}\wedge T$ converges weakly to $u(dd^{c}u)^{q}\wedge \beta^{n-m}\wedge T$, which implies the statement for $k=q+1$. Thanks to the quasi-continuity of $u$ with respect to $cap_{m,T}$ (see \cite{ref9}), for each $\varepsilon>0$, $u$ can be written as $u =\phi + \psi$
on $\Omega$, where  $ \phi$ is continuous,  $\psi = 0$ outside an open subset $G\subset \Omega$
 with  $cap_{m,T}(G)< \varepsilon$, and the supremum norm of $ \psi $ depends only on the function $u$. We have $$\begin{array}{lcl}
 u_{j}(dd^{c}u_{j})^{q}\wedge\beta^{n-m}\wedge T-u(dd^{c}u)^{q}\wedge\beta^{n-m}\wedge T&=&(u_{j}-u)(dd^{c}u_{j})^{q}\wedge\beta^{n-m}\wedge T\\&+& \psi\left[(dd^{c}u_{j})^{q}\wedge\beta^{n-m}\wedge T-(dd^{c}u)^{q}\wedge\beta^{n-m}\wedge T\right]\\&+& \phi\left[(dd^{c}u_{j})^{q}\wedge\beta^{n-m}\wedge T-(dd^{c}u)^{q}\wedge\beta^{n-m}\wedge T\right]\\&=&A_j+B_j+C_j.\end{array}$$
 The inductive assumption gives that $C_j$ converges to
$0$ in the sense of currents. On the other hand, it is not hard to see that $$(dd^{c}u_{j})^{q}\wedge\beta^{m-p-q}\wedge\beta^{n-m}\wedge T=(dd^{c}u_{j})^{q}\wedge\beta^{m-p-q-1}\wedge\beta^{n-m+1}\wedge T\leqslant \left(dd^{c}(u_{j}+|z|^2)\right)^{m-p-1}\wedge\beta^{n-m+1}\wedge T.$$ As the sequence $u_j+|z|^2$ is uniformly bounded on $\Omega$, it is clear that the last term is dominated by $cap_{m-1,T}$ multiplied by a constant not depending on $j$. Then, in view of the argument used in the case $k=1$, it is not difficult to deduce that $A_j$ also converges to $0$ in the sense of currents. Similarly, since $\psi=0$ outside $G$, for a test form $\varphi$, there exist two constants $C_1,C_2>0$ not depending on $j$ such that  $$ \begin{array}{lcl}\left|\left\langle B_j,\varphi\right\rangle\right|&\leqslant& C_1\|u\|_{L^{\infty}(G)}\left(\ds\int_{G} \left[(dd^{c}u_{j})^{q}\wedge\beta^{m-p-q} \wedge\beta^{n-m}\wedge T-(dd^{c}u)^{q}\wedge\beta^{m-p-q}\wedge\beta^{n-m}\wedge T\right]\right)\\&\leqslant&C_2\|u\|_{L^{\infty}(G)}cap_{m-1,T}(G)<\varepsilon C_2\|u\|_{L^{\infty}(G)}.
\end{array}$$
Consequently, $B_j$ converges to $0$ in the sense of the currents and therefore, we have obtained the weak convergence of $u_{j}(dd^{c}u_{j})^{q}\wedge\beta^{n-m}\wedge T$ to $u(dd^{c}u)^{q}\wedge\beta^{n-m}\wedge T  $ and this proof is complete.
\end{proof}
  Now, we investigate the convergence in the sense of $cap_{m-1,T}$ for increasing sequence of $m$-sh functions which are not  necessarily bounded.
 \begin{pro}
Assume that $ T\in{\mathscr C}_p^m(\Omega)$ and $u,u_{j}$ are negative $m$-sh functions such that $u_{j}\uparrow u\  in\  cap_{m,T}$-a.e and $ u_{1}\in L^1_{loc}(\sigma_T)$. Then $u_{j}$ converges to $u$ in $cap_{m-1,T}$ on every $E\Subset \Omega,$
\end{pro}
\begin{rem}\
\begin{enumerate}
\item For the trivial current $T=1$ and by Theorem 3.9 in \cite{ref12} if $u,u_{j}\in {\mathscr E}^m(\Omega)$ such that $u_{j}\uparrow u$, then $(dd^cu_j)^m\wedge\beta^{n-m}$ converges weakly to $(dd^cu)^m\wedge\beta^{n-m}$ when $j\rightarrow +\infty$.
\item If $T\in{\mathscr C}_p^m(\Omega)$ such that $m\geqslant p+1,$ $u\in {\cal {SH}}_m(\Omega)\cap L^{\infty}_{loc}(\Omega)$ and $u_{j}\in {\cal {SH}}_m(\Omega)$ such that $(u_{j})_j$ is locally uniformly bounded and increases $cap_{m,T}$-a.e to $u$, then the sequence $(dd^c u_{j})^{m-p}\wedge\beta^{n-m}\wedge T$ converges
weakly to $ (dd^c u)^{m-p}\wedge\beta^{n-m}\wedge T $. In the border case $m=n$, we recover the well-known result of Bedford-Taylor.
\end{enumerate}\end{rem}
\begin{proof} Let $E\Subset \Omega$ and $v\in{\cal {SH}}_m(\Omega)$ such that $-1\leqslant v\leqslant0$. Choose $s\gg1,$ we have: $$  \begin{array}{lcl} \ds\int\limits_{E\cap\{u-u_{j}\geqslant\delta\}}(dd^{c}v)^{m-p-1}\wedge\beta^{n-m+1}\wedge T&=&\ds\int\limits_{E\cap\{u-u_{j}\geqslant\delta\}\cap\{u_{1}<-s\}}(dd^{c}v)^{m-p-1}\wedge\beta^{n-m+1}\wedge T\\&+&\ds\int\limits_{E\cap\{u-u_{j}\geqslant\delta\}\cap\{u_{1}\geqslant -s\}} (dd^{c}v)^{m-p-1}\wedge\beta^{n-m+1}\wedge T \\&=& I_{1}+I_{2}\end{array}$$
In view of the hypothesis on $u_1,$ the Chern-Levine-Nirenberg inequality \cite{ref13} implies that there exists a constant $C_1>0$ not depending in $v$ such that:
$$ I_1\leqslant \ds\int\limits_{E\cap\{u_{1}<-s\}}(dd^{c}v)^{m-p-1}\wedge\beta^{n-m+1}\wedge T \leqslant\frac{C_1}{s}.$$
Concerning $I_2$, setting $u^{s}=\max(u,-s)$ and $u_{j}^{s}=\max(u_{j},-s)$ and observe that for a fixed $s$, $u^{s}$ and $u^{s}_{j}$ are locally uniformly bounded. Hence, by using a standard modification on the functions $u^s,u_j^s$ (see \cite{ref2}), we can assume that for any $j,$ $ u_{j}^s=u^s$ near $\partial\Omega.$ Therefore, we have
 $$\begin{array}{lcl}  I_{2}&=&\ds\int\limits_{E\cap\{u^s-u_{j}^s\geqslant\delta\}\cap\{u_{1}\geqslant -s\}} (dd^{c}v)^{m-p-1}\wedge\beta^{n-m+1}\wedge T \\&\leqslant& \frac{1}{\delta}\ds\int\limits_{E\cap\{u_{1}\geqslant -s\}}(u^{s}-u_{j}^{s})(dd^{c}v)^{m-p-1}\wedge\beta^{n-m+1}\wedge T \\&\leqslant
&\frac{1}{\delta}\ds\int_{\Omega}(u^{s}-u_{j}^{s})(dd^{c}v)^{m-p-1}\wedge\beta^{n-m+1}\wedge T. \end{array}$$  By following the same line of the proof of Lemma 2 in \cite{ref9}, we obtain $$\begin{array}{lcl} I_{2}&\leqslant&\frac{1}{\delta} \ds\int_{\Omega}(u^{s}-u_{j}^{s})(dd^{c}v)^{m-p-1}\wedge\beta^{n-m+1}\wedge T\\&\leqslant&\frac{C_{2}}{\delta}\ds\int_{\Omega}(u^{s}-u_{j}^{s})(dd^{c}u_{j}^{s})^{m-p-1}\wedge\beta^{n-m+1}\wedge T \\&=&\frac{C_{2}}{\delta}\ds\int_{\Omega}|z|^{2} (dd^{c}u^{s}-dd^{c}u^{s}_{j})\wedge (dd^{c}u_{j}^{s})^{m-p-1}\wedge\beta^{n-m}\wedge T\\&=&\frac{C_{2}}{\delta}\ds\int_{\Omega}|z|^{2}[dd^{c}u^{s}\wedge(dd^{c}u_{j}^{s})^{m-p-1}\wedge\beta^{n-m}\wedge T- (dd^{c}u_{j}^{s})^{m-p}\wedge \beta^{n-m}\wedge T], \end{array}$$where $C_2>0$ not depending on $v$ and $j$. By passing on the supremum on $v$, we get $$cap_{m-1,T}(E\cap\{u-u_{j}\geqslant\delta\})\leqslant\frac{C_1}{s}+\frac{C_2}{\delta}\int_{\Omega}|z|^{2}[dd^{c}u^{s}\wedge(dd^{c}u_{j}^{s})^{m-p-1}\wedge\beta^{n-m}\wedge T- (dd^{c}u_{j}^{s})^{m-p}\wedge \beta^{n-m}\wedge T].$$ On the other hand by the second statement of Remark 2, we have
$$\lim\limits_{j\rightarrow +\infty} dd^{c}u^{s}\wedge(dd^{c}u_{j}^{s})^{m-p-1}\wedge\beta^{n-m}\wedge T =\lim\limits_{j\rightarrow+\infty}(dd^{c}u_{j}^{s})^{m-p}\wedge \beta^{n-m}\wedge T = (dd^{c}u^{s})^{m-p}\wedge\beta^{n-m}\wedge T. $$ The proof was completed by passing to the limit when $j\rightarrow +\infty$ and $s\rightarrow +\infty$ in this order. \end{proof}
 In Proposition 5, for $T=1$, if we replace the monotonicity hypothesis by some conditions on the estimates of the complex Hessian  operator, we get a slightly stronger convergence, that is a convergence in $cap_m$ on the class $\mathscr{F}^{m}(\Omega)$. By an adaptation of the proof of Theorem 5.1 in \cite{ref16} in the complex case $m=n$, we obtain the following generalization to the complex Hessian case:
\begin{pro}
Let $u,u_{j}\in{\mathscr F}^{m}(\Omega)$ such that $u_{j}\leqslant u$ for any $j\geqslant 1.$  Suppose that
\begin{enumerate}
\item $
\sup_{j}\int_{\Omega} (dd^{c}u_{j})^{m}\wedge\beta^{n-m}<\infty .$
\item $\Vert(dd^{c}u_{j})^{m}\wedge \beta^{n-m}-(dd^{c}u)^{m}\wedge\beta^{n-m}\Vert_{E}\rightarrow 0$ when $ j\rightarrow +\infty $ and for any $E\Subset\Omega.$
\end{enumerate} Then $u_{j}\rightarrow u$ in $m$-capacity on every $E\Subset \Omega$ when $j\rightarrow +\infty$.
\end{pro}
\section{Weighted $(m,T)$-capacity and Sadullaev weighted $m$-capacity}
The notions of weighted capacity and weighted extremal function in the pluripotential theory, were studied by many authors (see for example \cite{ref7},\cite{ref8} and \cite{ref15}). In the complex Hessian context, we define the weighted $m$-extremal function as follows:
\begin{defn}
Let  $\Omega$ be an  $m$-hyperconvex domain, $ E\subset\Omega $ and  $ u\in{\cal {SH}}_m^-(\Omega) $. Setting
$$ h_{m,E,u}=\sup \lbrace v:v\ \in {\cal {SH}}_{m}^{-}(\Omega),\ v \leqslant u\ {\rm on}\ E\rbrace . $$
The function $ h_{m,E,u} $ is called \textit{$m$-extremal function with weight $u$} associated to $E$.
\end{defn}
 As usually, we denote by $ h^{\ast}_{m,E,u} $ the regularization upper semi-continuous of $ h_{m,E,u}.$ It is clear that  $ h_{m,E,u}\leqslant h_{m,E,u}^{\ast}\in{\cal {SH}}_{m}^{-}(\Omega) $ and $u\leqslant h_{m,E,u}^\ast\leqslant 0$ and $h_{m,E,u}=u$ on $E$. If $ m=n $, we recover the weighted extremal function investigated by \cite{ref8},\cite{ref15} and if  $1\leqslant m < n\ and\  u=-1 $ we obtain the $m$-extremal function developed in \cite{Bl},\cite{Den}, \cite{ref13} and \cite{ref17}. First of all we list the most important properties of the weighted $m$-extremal function.
 \begin{pro} The function $h_{m,E,u}^\ast$ shares the following properties:
\begin{enumerate}
\item $ h^{\ast}_{m,E,u}=h^{\ast}_{m,E\smallsetminus F,u} $ for all $m$-polar $ F\subset\Omega $.
\item If $u\in{\mathscr E}^m(\Omega)$ then $h^{\ast}_{m,E,u}\in \mathscr{E}^{m}(\Omega)$ and  $h^{\ast}_{m,E,u}= h_{m,E,u} $ a.e in $m$-capacity. Moreover, we have $\Supp[(dd^{c}h^{\ast}_{m,E,u})^{m}\wedge\beta^{n-m}]\subset\overline{E} .$
\item If $(u_{j})_j\subset{\cal {SH}}_m^-(\Omega) $ such that $ u_{j}\downarrow u$  then $ h_{m,E,u_{j}}^{\ast}\downarrow  h_{m,E,u}^{\ast}.$
\item If $u\in{\mathscr E}^m(\Omega)$ and $E\Subset\Omega$ then $h_{m,E,u}^\ast\in{\mathscr F}^m(\Omega)$.
\end{enumerate}
\end{pro}
\begin{proof} The last three statements are due to \cite{Van Nguyen}, so it suffices to prove the first one.  Since $E\smallsetminus F\subset E$, it is obvious that $ h^{\ast}_{m,E,u}\leqslant h^{\ast}_{m,E\smallsetminus F,u}$. For the opposite inequality, let $\varphi$ be a negative $m$-sh function on $\cb^n$ such that $F\subset\{\varphi=-\infty\}$ and let $v\in{\cal {SH}}_{m}^{-}(\Omega)$ such that $v\leqslant u$ on $E\smallsetminus F$. It is clear that the sequence of $m$-sh functions $v_j=v+\frac{\varphi}{j}$ is monotone increasing and satisfies $v_j\leqslant u$ on $E$ for every $j\geqslant 1$. It follows that $v_j\leqslant h_{m,E,u} \ \forall j\geqslant 1$. Hence, $\widetilde v=\left(\sup_j v_j\right)^\ast\leqslant h^\ast_{m,E,u}$. Since $\widetilde v$ is $m$-sh and coincides a.e with $v$ we get $v\leqslant h^\ast_{m,E,u}$ on $\Omega$ and therefore $h_{m,E\smallsetminus F,u}\leqslant h_{m,E,u}^\ast$ on $\Omega$. The desired inequality follows because $h_{m,E\smallsetminus F,u}=h_{m,E\smallsetminus F,u}^\ast$ a.e.
\end{proof}
It should be noted that Nguyen \cite{Van Nguyen} has introduced the weighted $m$-extremal function by setting: $u_E=\sup \lbrace v:v\ \in {\cal {SH}}_{m}^{-}(\Omega),\ v \leqslant u$ outside an $m$-polar subset of $ E\rbrace . $
Building on the first statement of Proposition 7, it is not difficult to prove that the two definitions are coincide, in particular we have $h_{m,E,u}^\ast=u_E$. In this section by relying on Proposition 7, we firstly associate to each current $T\in{\mathscr C}_p^m(\Omega)$ two classes of unbounded $m$-sh functions on which the complex hessian operator is well defined for the current $T=1$. In particular, for the complex setting $m=n$ we recover some interesting classes studied by Bedford \cite{ref3}, Cegrell-Kolodziej-Zeriahi \cite{ref8} and Benelkourchi \cite{ref4} in connection with the complex Monge-Amp\`ere operator. Next, we introduce the notion of $(m,T)$-capacity with weight $u\in{\cal {SH}}_m^-(\Omega)$, where $T\in{\mathscr C}_p^m(\Omega)$. Moreover, in the special case $T=1$ we investigate the relationship between the weighted outer capacity associated to the weighted $(m,1)$-capacity and the weighted $m$-extremal function. We generalize then some results obtained by \cite{ref8} and \cite{Van Nguyen}. Finally, we deal with the so-called Sadullaev weighted $m$-capacity and we present a comparison result with the weighted outer $m$-capacity.
\subsection{ The classes ${\mathscr B}_{T}^{m}(\Omega)$ and ${\mathscr P}^{m}_{T}(\Omega) $}
 Let $ \varphi:\rb\rightarrow\rb_{+} $ be a decreasing function such that
\begin{equation} \int_{1}^{\infty}\frac{\varphi(t)}{t}dt<+\infty
\end{equation}
and the function $t\mapsto -(-t\varphi(-t))^{\frac{1}{m}} $ is increasing and convex on $ ]-\infty,0[.$  We introduce  $ {\mathscr B}_{T}^{m}(\Omega) $ to be the class of negative $m$-sh functions $ \psi $ such that for every $ z_{0}\in\Omega,$ there exist an open neighborhood ${\mathscr O}$ of $z_{0},$ a function $v\in {\cal {SH}}_{m}^{-}({\mathscr O})\cap L^{1}_{loc}(\sigma_{T})$  and a function $\varphi $ satisfying $(4.1)$ such that $-(-v\varphi(-v))^{\frac{1}{m}}\leqslant \psi$  on ${\mathscr O}\cap\Supp T.$ For $T=1$ and $m=n,$ Cegrell-Kolodziej-Zeriahi \cite{ref8} proved that $ {\mathscr B}_1(\Omega)\subset{\mathscr E}(\Omega).$ In particular, the complex Monge-Amp\`ere is well defined on ${\mathscr B}_1(\Omega)$.
\begin{exe} Let $ T\in{\mathscr C}_p^m(\Omega)$ and $v\in {\cal {SH}}_{m}^{-}(\Omega)\cap L^{1}_{loc}(\sigma_{T}).$ Then we have $-(-v)^{\alpha}\in {\mathscr B}^{m}_{T}(\Omega),$ for  each $0<\alpha<\frac{1}{m}$. Indeed,
 we have $-(-v)^{\alpha}=-(-v\varphi(-v)) ^{\frac{1}{m}}$ for $\varphi(t)=t^{m\alpha-1}$. Moreover $\int^{\infty}_{1} \frac{t^{m\alpha-1}}{t}dt=\frac{1}{1-m\alpha}<\infty$ and $t\mapsto -(-t\varphi(-t))^{\frac{1}{m}}$ is increases and  convex on $ ]-\infty,0[. $ Then by the definition, it is clear that $-(-v)^{\alpha}\in{\mathscr B}^{m}_{T}(\Omega).$
\end{exe}
For every $1\leqslant m\leqslant n ,$ we define the second class ${\mathscr P}^m_T(\Omega)$ by:
$${\mathscr P}^{m}_{T}(\Omega)= \left\{ u\in {\cal {SH}}_{m}^{-}(\Omega);\   \int_{0}^{+\infty}s^{m-1}cap_{m,T}(\{u<-s\}\cap K)ds <\infty,\ \ \forall  K\cap\Supp T\Subset \Omega \right\}.$$
 By using some potential theoretic estimates relatively to $T$ inspired from \cite{ref13} for the trivial current $T=1$, we show that there is a link between the two above classes. More precisely, we prove the following result which improves the one obtained by \cite{ref4} for $m=n$ and $T=1$:
\begin{pro}
Let $\Omega \Subset \cb^{n}$, be an $m$-hyperconvex domain and $ T\in{\mathscr C}_p^m(\Omega)$ then we have ${\mathscr B}^{m}_{T}(\Omega)\subset {\mathscr P}^{m}_{T}(\Omega).$
 In particular, when $T=1,$ we have ${\mathscr P}^{m}_{1}(\Omega)\subset {\mathscr E}^{m}(\Omega)$ i.e, the complex Hessian operator is well defined on ${\mathscr P}^{m}_{1}(\Omega)$.
\end{pro}
\begin{proof}
 Let $u\in{\mathscr B}^{m}_{T}(\Omega)$ and ${\mathscr O}\Subset\Omega$. Then there exist $v\in{\cal {SH}}_{m}^{-}({\mathscr O})\cap L_{loc}^{1}(\sigma_{T})$ and a decreasing function $\varphi:\rb\rightarrow\rb_{+}$ satisfynig $(4.1)$ and such that $ t\mapsto-(-t\varphi(-t))^{\frac{1}{m}}$ is monotone increasing and convex on $ ]-\infty,0[ $  and $ u\geqslant -(-v\varphi(-v))^{\frac{1}{m}} $ on ${\mathscr O}\cap\Supp T.$ We have:  $$ \{u<-s\}\cap{\mathscr O}\cap\Supp T\subset \{-(-v\varphi(-v))^{\frac{1}{m}}< -s\}\cap\Supp T=\{-v\varphi(-v)> s^{m}\}\cap\Supp T.$$
Choose a function $ \psi $ such that $\psi^{'}=\varphi$ and $\psi(0)=0$. Since $\psi$
is concave we see that $\psi(t)\geqslant t \varphi(t)$ for all $t>0.$
In view of the later inclusion, we have $$\begin{array}{lcl} \ds\int_{0}^{+\infty}s^{m-1}cap_{m,T}(\{u<-s\}\cap {\mathscr O})\  ds&\leqslant&\ds\int_{0}^{+\infty}s^{m-1}cap_{m,T}(\{-v\varphi(-
v)>s^{m}\})\  ds \\ &\leqslant&\ds\int_{0}^{+\infty}s^{m-1} cap_{m,T}(\{\psi(-v)>s^{m}\})\ ds \\&=&\ds\int_{0}^{+\infty}s^{m-1} cap_{m,T}(\{v< -\psi^{-1}(s^{m})\} )\ ds \\ &\leqslant& C_{1}+C_{2}\ds\int_{1}^{+\infty}s^{m-1} \frac{1}{\psi^{-1}(s^{m})}\ ds\\&=&C_{1}+C_{2}\ds\int_{1}^{+\infty}\frac{\varphi(t)}{t}\ dt<\infty.
\end{array} $$
The third inequality is deduced from the following argument: taking into account the fact that ${\mathscr O}\Subset\Omega$, and $v\in{\cal {SH}}_{m}^{-}({\mathscr O})\cap L_{loc}^{1}(\sigma_{T})$ and combining the proof of Proposition 1.3.4 with inequality 1.2.8 in the same thesis \cite{ref13}. Concerning the later equality we use a change of variables. This completes the statement $ u\in{\mathscr P}^{m}_{T}(\Omega).$ Now, assume that $T=1$ and let $ u\in {\mathscr P}_{1}^{m}(\Omega)$. Since $ u\in{\cal {SH}}_{m}^{-}(\Omega),$ there exists a decreasing sequence $(u_{j})_{j}\subset{\mathscr E}_{0}^{m}(\Omega){\cap \mathscr {C}(\overline{\Omega})}$ to $u$. Take ${\mathscr B}={\mathscr B}_{z_0}(r)\Subset \Omega$ and for every $j\geqslant 1$ we consider the weighted extremal function $h_{m,{\mathscr B},u_{j}}$ relatively to $u_j$.  In virtue of Proposition 7, the sequence $ (h^{\ast}_{m,{\mathscr B},u_{j}})_{j} $ is monotone decreasing to $h^{\ast}_{m,{\mathscr B},u}$. We begin by showing that  $\sup_{j}\int_{\Omega} (dd^{c}h^{\ast}_{m,{\mathscr B},u_{j}})^{m}\wedge\beta^{n-m}<\infty.$ This is an easy consequence of an estimates of the complex Hessian operator in terms the $m$-capacity. In other word, it suffices to prove the existence of a constant $ C$ such that for every function $ \varphi\in {\cal {SH}}_{m}^{-}(\Omega)\cap L^{\infty}(\Omega), $ we have $$ \int_{K} (dd^{c}\varphi)^{m}\wedge\beta^{n-m}\leqslant C \int_{0}^{+\infty} s^{m-1} cap_{m}(K\cap\{\varphi\leqslant-s\})ds,\quad \forall K\Subset\Omega .$$
For this aim, a straightforward computation yields  $$\begin{array}{lcl}\ds\int_{K} (dd^{c}\varphi)^{m}\wedge\beta^{n-m}&=&\sum\limits_{-\infty}^{+\infty}\ds\int_{K\cap \{2^{k-1}\leqslant-\varphi<2^{k}\}}(dd^{c}\varphi)^{m}\wedge\beta^{n-m}\\&\leqslant&\sum\limits_{-\infty}^{+\infty}2^{km}cap_{m}(K\cap\{2^{k-1}\leqslant-\varphi< 2^{k}\})\\&\leqslant&C\sum\limits_{-\infty}^{+\infty}\ds\int_{2^{k-2}}^{2^{k-1}}s^{m-1}cap_{m}(K\cap\{-\varphi\geqslant s\}) ds\\&\leqslant&C\ds\int_{0}^{\infty}s^{m-1}cap_{m}(K\cap\{-\varphi\geqslant s\})  ds
\\&=& C\ds\int_{0}^{\infty}s^{m-1}cap_{m}(K\cap\{\varphi\leqslant -s\}) ds. \end{array}$$
Now, Proposition 7 with the above estimates for each $h^{\ast}_{m,{\mathscr B},u_{j}}$, yield
$$\begin{array}{lcl} \ds\int_{\Omega}(dd^{c}h^{\ast}_{m,{\mathscr B},u_{j}})^{m}\wedge\beta^{n-m}=\ds\int_{\overline{{\mathscr B}}}(dd^{c}h^{\ast}_{m,{\mathscr B},u_{j}})^{m}\wedge\beta^{n-m}&\leqslant& C\ds\int_{0}^{\infty}s^{m-1}cap_{m}(\overline{{\mathscr B}}\cap\{h^{\ast}_{m,{\mathscr B},u}\leqslant -s\})\ ds\\&\leqslant& C\ds\int_{0}^{\infty}s^{m-1}cap_{m}(\overline{{\mathscr B}}\cap\{u\leqslant -s\})\ ds <\infty.\end{array} $$
The first inequality because $ h^{\ast}_{m,{\mathscr B},u}\leqslant h^{\ast}_{m,{\mathscr B},u_j}$ while the third one is an immediate consequence of the fact that $u\in {\mathscr P}_{1}^{m}(\Omega).$ This leads to the statement $h^{\ast}_{m,{\mathscr B},u}\in{\mathscr F}^{m}(\Omega)$. Thus the proof was completed by observing that $h^{\ast}_{m,{\mathscr B},u}=u$ on  ${\mathscr B}$ and by using Remark 1.7.6 in \cite{ref13}.   \end{proof}
\subsection{Weighted $(m,T)$-capacity}
Let $ T\in{\mathscr C}_p^m(\Omega)$, $K$ a compact subset of $\Omega $ and $ u\in {\cal {SH}}_m^-(\Omega)$. We define the $(m,T)$-capacity with weight $u$ of $K$ by $$ cap_{m,T,u}(K)=\sup\left\{\int_{K}(dd^{c}v)^{m-p} \wedge\beta^{n-m}\wedge T ;\ v \in {\cal {SH}}_{m}(\Omega) \cap L^{\infty}(\Omega),\ u\leqslant v \leqslant 0\right\},$$ and for every $E\subset\Omega$, $cap_{m,T,u}(E)=\sup\{cap_{m,T,u}(K), K\subset E\ \rm{compact }\}$. If $E$ is Borel then $$cap_{m,T,u}(E)=\sup\left\{\int_{E}(dd^{c}v)^{m-p} \wedge\beta^{n-m}\wedge T ;\ v \in {\cal {SH}}_{m}(\Omega) \cap L^{\infty}(\Omega),\ u\leqslant v \leqslant 0\right\}.$$
In order to understanding this notion of weighted $(m,T)$-capacity let us recall the following comments:
\begin{rem} \begin{enumerate}
\item Unlike the case $u\equiv -1$, the quantity $cap_{m,T,u}(K)$ may be infinite even when $K$ is compact. However, by the Chern-Levine-Nirenberg inequality it is obvious that $cap_{m,T,u}(E)$ is finite provided that $u$ is locally bounded everywhere and $E\Subset\Omega$.
\item When $u=-1$, we get the capacity $cap_{m,T}$ studied by \cite{ref9} and if in addition $T=1$, we obtain the $m$-capacity, $cap_{m}$ of Lu \cite{lu article}. However, according to the terminology of \cite{Van Nguyen}, the weighted $(m,1)$-capacity is called weighted $m$-capacity and denoted by $cap_{m,u}$. When $m=n$ and $u\in\mathscr E(\Omega)$, we recover the weighted capacity investigated by \cite{ref15} and \cite{ref8}.
\item A simple computation yields that the weighted $(m,T)$-capacity shares the same properties of $cap_{m,T}$ as in Proposition 3. Moreover, clearly $cap_{m,T,u}$ is monotone decreasing with respect to $u$. In particular, if $u\geqslant -M$ for $M>0$ then we see that $\int_{E}(dd^{c}u)^{m-p} \wedge\beta^{n-m}\wedge T\leqslant cap_{m,T,u}(E)\leqslant cap_{m,T,-M}(E)=M^{m-p}cap_{m,T}(E)$ for every Borel subset $E\subset\Omega$.
\item Assume that $0\in\Omega$ and let $L$ be a complex linear space of dimension $p$ in $\cb^n$. By a unitary change of coordinates, we assume that $L=\cb^p\times\{0\}$. Let ${\mathscr O}$ be an open subset of $\Omega$ and $u\in{\cal {SH}}_m^-(\Omega)$. Select a bounded $m$-sh function $v$ such that $u\leqslant v\leqslant 0$. Then
$$\ds\int_{{\mathscr O}}
[L]\wedge\beta^{n-m}\wedge(dd^c v)^{p+m-n}=\ds\int_{{\mathscr O}\cap L}(i^{\star}\beta)^{p-(m+p-n)}\wedge(dd^c (i^{\star}v))^{p+m-n}.$$
By considering the current of integration on $L$, $T=[L]$ and using the fact that $u_{|{L\cap\Omega}}=i^{\star}u$ is $(m+p-n)$-sh, see \cite{ref17}, we deduce that the $(m,[L])$-weighted capacity of ${\mathscr O}$ on $\Omega$ with weight $u$ is nothing but the $m$-weighted capacity of $L\cap {\mathscr O}$ on $L\cap\Omega$ with weight $i^{\star}u$.
\end{enumerate}
\end{rem}
We now extend to this weighted $(m,T)$-capacity a monotonicity result due respectively to \cite{ref15} for $m=n$ and $T=1$ and to \cite{Van Nguyen} for $T=1$ and $m<n$.
\begin{pro}
Assume that $u_{j} ,u\subset {\cal {SH}}_m^-(\Omega)$ such that $(u_{j})_j$ is monotone decreasing to $u$. Let $T\in{\mathscr C}_p^m(\Omega),$ such that $\ds\lim_{s\l+\infty}cap_{m,T}(\{u<-s\})=0$. Then, $cap_{m,T,u_{j}}(E)$ is monotone increasing to $cap_{m,T,u}(E),$ for any subset $E\subset\Omega$ such that $cap_{m,T,u}(E)<+\infty$.\end{pro}
\begin{proof} Take $E\subset\Omega$ such that $cap_{m,T,u}(E)<+\infty$ and let $K\subset E$ be a compact set. By monotonicity of $(u_{j})_{j}$ and in view of the definition of the weighted $(m,T)$-capacity, it is clear that $ cap_{m,T,u_{j}}(K)\leqslant cap_{m,T,u_{j+1}}(K)\leqslant cap_{m,T,u}(K)<+\infty,\ \forall \emph{ } j\geqslant 1 $. Assume that $\alpha$ is a limit of $ cap_{m,T,u_{j}}(K)$, then it suffices to prove that $ \alpha \geqslant cap_{m,T,u}(K) $. For this goal, let $ \varphi \in {\cal {SH}}_{m}(\Omega)\cap L^{\infty}(\Omega) $ such that $ u\leqslant \varphi \leqslant 0 $. Since $\{u_j<-s\}\subset\{u<-s\}$, for every $j$, by the hypothesis for $\varepsilon>0$, we can find an open set ${\mathscr O}_1$ such that $cap_{m,T}({\mathscr O}_1)<\varepsilon/3$ and the functions $u_j,u$ are locally bounded on $\Omega\smallsetminus{\mathscr O}_1$. On the other hand, thanks to the quasicontinuity Theorem with respect to $cap_{m,T}$ (see \cite{ref9}), there exists an open set $ {\mathscr O}_2\subset\Omega,$ such that $ cap_{m,T}({\mathscr O}_2)< \frac{\varepsilon}{3} $ and $ u_{|{\Omega\smallsetminus {\mathscr O}_2}} $ is continuous. By the same raison, $ \forall j\geqslant 1 , \exists$ ${\mathscr O}_2^{j}\subset\Omega $ such that $ cap_{m,T}({\mathscr O}_2^{j})< \frac{\varepsilon}{3.2^{j+1}} $ and $ {u_j}_{|{\Omega\smallsetminus{\mathscr O}_2^{j}}} $ is continuous. Setting ${\mathscr O}={\mathscr O}_1\cup{\mathscr O}_2\cup \cup_{j=1}^{\infty}{\mathscr O}_2^{j} $ then since $cap_{m,T}$ is sub-additive we have $cap_{m,T}({\mathscr O})<\varepsilon $, $ {u_j}_{|\Omega\smallsetminus {\mathscr O}} $ and $ u_{|\Omega\smallsetminus{\mathscr O}} $ are continuous. Let $K\subset \Omega^{'}\Subset \Omega $, by the Dini's Theorem  $ u_{j} $ converges uniformly to $u$ on $ \overline{\Omega^{'}}\smallsetminus{\mathscr O} .$ So, there exists $ j_{0}\in\nb^\star $, such that on $\overline{\Omega^{'}}\smallsetminus{\mathscr O}$, we have $ u_{j_{0}}<(1-\varepsilon)u\leqslant (1-\varepsilon)\varphi.$ Thus, we have
$$\begin{array}{lcl} \alpha\geqslant cap_{m,T,u_{j_{0}}}(K)&\geqslant& \ds\int_{K}(dd^{c}\max ((1-\varepsilon)\varphi,u_{j_{0}}))^{m-p}\wedge\beta^{n-m}\wedge T\\ &\geqslant& \ds\int_{K\smallsetminus{\mathscr O}}(dd^{c}\max ((1-\varepsilon)\varphi,u_{j_{0}}))^{m-p}\wedge\beta^{n-m}\wedge T\\ &=& (1-\varepsilon)^{m-p}\ds\int_{K\smallsetminus{\mathscr O}}(dd^{c}\varphi)^{m-p}\wedge\beta^{n-m}\wedge T \\ &=&(1-\varepsilon)^{m-p}\left[\ds\int_{K}(dd^{c}\varphi)^{m-p}\wedge\beta^{n-m}\wedge T-\int_{{\mathscr O}}(dd^{c}\varphi)^{m-p}\wedge\beta^{n-m}\wedge T\right]\\ &\geqslant& (1-\varepsilon)^{m-p} \left[\ds\int_{K}(dd^{c}\varphi)^{m-p}\wedge\beta^{n-m}\wedge T- (\ds\sup_{{\mathscr O}}|\varphi|)^{m-p} cap_{m,T}({\mathscr O})\right ] \\ &\geqslant& (1-\varepsilon)^{m-p} \left[\ds\int_{K}(dd^{c}\varphi)^{m-p}\wedge\beta^{n-m}\wedge T-\varepsilon (\ds\sup_{{\mathscr O}}|\varphi| )^{m-p}\right].\end{array}$$
When $ \varepsilon\rightarrow 0 $, we deduce the inequality $\alpha\geqslant cap_{m,T,u_{j_{0}}}(K)\geqslant \int_{K}(dd^{c}\varphi)^{m-p}\wedge\beta^{n-m}\wedge T $ and therefore by arbitrariness of $\varphi$ and $K$, we easily see that $ \alpha\geqslant cap_{m,T,u}(E) $.\end{proof}
For two $m$-positive currents $S$ and $R$, we say that $S\prec R$ if and only if $R-S$ is $m$-positive. Recall that for a subset $E\subset\Omega$, the Hausdorff $s$-content of $E$ is defined by
$$\widehat{\mathscr H}_s(E)={\rm inf}\ds\sum_j r_j^s ,$$
where the infimum is taken over all coverings of $E$ by balls with radius $r_j$. It is not hard to see that $\widehat{\mathscr H}_s\leqslant {\mathscr H}_s$, where ${\mathscr H}_s$ is the Hausdorff $s$-measure. According to \cite{ref20}, it was proved that the capacity $cap_{m,T}$  can be controlled locally by the Hausdorff $\frac{2np}{m}$-content provided that $T\in{\mathscr C}_p^m(\Omega)$ and $T\prec\beta^p$. Here, we present the following similar result for weighted $(m,T)$-capacity:
\begin{pro} Assume that $u\in{\cal {SH}}_m^-(\Omega)\cap{\mathscr F}^{m-p}(\Omega)$ and $T\in{\mathscr C}_p^m(\Omega)$ such that $T\prec\beta^p,$ $n\geqslant m\geqslant p+1$. Then, for any $E\Subset\Omega$ there exists $A=A(E,m,n,p)>0$ such that $cap_{m,T,u}(F)\leqslant A\left(\int_\Omega(dd^cu)^{m-p}\wedge\beta^{n-m+p}\right)\widehat{\mathscr H}_{\frac{2np}{m}}(F)$, for any $F\Subset E$. In particular, if $\widehat{\mathscr H}_{\frac{2np}{m}}(F)=0$, then $F$ is not charged by $cap_{m,T,u}$.
\end{pro}
\begin{proof} We use a method from \cite{ref20}. For any $\varepsilon>0$, take a cover of $F$ by balls $\mathscr B_{r_j}(a_j)$ such that $$\sum_j r_j^{\frac{2np}{m}}<\widehat{\mathscr H}_{\frac{2np}{m}}(F)+\varepsilon.$$
 Let $E_{r}=\{z\in\Omega:\ d(z,E)<r\}$ and choose $\delta>0$, so that $E_{2\delta}\Subset\Omega$. Observe that in the preceding cover of $F$, we select only balls $\mathscr B_{r_j}(a_j)$ such that $r_j<\delta$ because in that case we have $a_j\in E_\delta$, otherwise we do not consider such ball and the cover is still hold. Let $v$ be a bounded $m$-sh function with $u\leqslant v\leqslant 0$. Since $T\prec\beta^p,$ it is clear that $(dd^c v)^{m-p}\wedge\beta^{n-m}\wedge T\leqslant(dd^c v)^{m-p}\wedge\beta^{n-m+p}$ as positive measures. On the other hand, by \cite{ref20} the function $r\mapsto r^{-\frac{2np}{m}}\int_{\mathscr B_{r}(a)}(dd^c v)^{m-p}\wedge\beta^{n-m+p}$ is monotone increasing. Hence, for each ball $\mathscr B_{r_j}(a_j)$ such that $r_j<\delta$, one has
$$
\begin{array}{lcl} \ds\int_{F\cap \mathscr B_{r_j}(a_j)}(dd^c v)^{m-p}\wedge\beta^{n-m}\wedge T&\leqslant&\ds\int_{F\cap \mathscr B_{r_j}(a_j)}(dd^c v)^{m-p}\wedge\beta^{n-m+p}\\&\leqslant& \frac{r_j^{\frac{2np}{m}}}{\delta^{\frac{2np}{m}}}\ds\int_{E_{2\delta}}(dd^cv)^{m-p}\wedge\beta^{n-m+p}\\&\leqslant&\frac{r_j^{\frac{2np}{m}}}{\delta^{\frac{2np}{m}}}\ds\int_{\Omega}(dd^cv)^{m-p}\wedge\beta^{n-m+p}\\&\leqslant&\frac{r_j^{\frac{2np}{m}}}{\delta^{\frac{2np}{m}}}\ds\int_\Omega(dd^cu)^{m-p}\wedge\beta^{n-m+p}.\end{array}$$
The second inequality is clearly satisfied when $r_j\geqslant\delta$ while the later one is deduced from Proposition 5.2 in \cite{ref12} because $u,\max(u,v)=v\in{\mathscr F}^{m-p}(\Omega)$. This implies that for any compact set $K\subset F$, we have
 $$
 \begin{array}{lcl}
 \ds\int_K(dd^cv)^{m-p}\wedge\beta^{n-m}\wedge T&\leqslant& \ds\sum_j\int_{F\cap {\mathscr B}_{r_j}(a_j)}(dd^c v)^{m-p}\wedge\beta^{n-m}\wedge T\\&\leqslant&\frac{1}{\delta^{\frac{2np}{m}}}\left(\ds\int_\Omega(dd^cu)^{m-p}\wedge\beta^{n-m+p}\right)\left(\ds\sum_jr_j^{\frac{2np}{m}}\right)\\&\leqslant& \frac{1}{\delta^{\frac{2np}{m}}}\left(\ds\int_\Omega(dd^cu)^{m-p}\wedge\beta^{n-m+p}\right)\left(\widehat{\mathscr H}_{\frac{2np}{m}}(F)+\varepsilon\right).
  \end{array}
 $$
Taking firstly the supremum over all such functions $v$, and secondly over all compact sets of $F$, we get the desired inequality.\end{proof}
Next we introduce the weighted outer $m$-capacity for a weight $u\in{\mathscr E}^m(\Omega)$ by setting
$$cap_{m,u}^\ast(E)=\inf\left\{cap_{m,u}({\mathscr O}) :\  {\mathscr O}\supset E,\ {\mathscr O}\ {\rm is\ open} \right\}.$$
It is not difficult to see that the most elementary properties of the weighted $m$-capacity are still hold for the outer weighted $m$-capacity. Especially, when the weight is identically $-1$, the weighted $m$-capacity is crucial in the characterization of $m$-polar subsets. Recently, in \cite{Van Nguyen} the author has proved that there is a direct relation between the weighted  $m$-capacity $ cap_{m,u} $ and the weighted $m$-extremal function $ h_{m,E,u}^{\ast}$ for any compact $E\subset \Omega$. Continuing this way, we get our first main comparison in this section:
\vskip0.15cm
\noindent{\bf Theorem C.} {\it Let $\Omega$ be a bounded $m$-hyperconvex open subset of $\cb^n$. We have
\begin{enumerate}
\item If $u\in {\mathscr E}^{m}(\Omega) $ then $$cap_{m,u}^\ast\left({\ds\mathop  E^\circ}\right)\leqslant\int_{\Omega}(dd^{c}h_{m,E,u}^{\ast})^{m}\wedge\beta^{n-m}\leqslant cap_{m,u}^\ast(\overline E),\quad for\ every\ E\Subset\Omega.$$
\item If $u\in{\mathscr F}^{m}(\Omega)$ then $$\int_\Omega(dd^ch_{m,E,u}^\ast)^m\wedge\beta^{n-m}\leqslant cap_{m,u}^\ast(E)\leqslant\int_\Omega(dd^cu)^m\wedge\beta^{n-m},\quad for\ every\ E\subset\Omega.$$
In particular, we have $cap_{m,u}(\Omega)=\int_\Omega(dd^cu)^m\wedge\beta^{n-m}.$
\end{enumerate}}
\vskip0.15cm
 Notice that estimates (1) is due to \cite{ref8} for the complex setting $m=n$. On the other hand, from (2) we see that if $E\Subset\Omega$ then $cap_m^\ast(E)$ coincides with the total weighted $m$-capacity associated to $h^\ast_{m,E,\Omega}$: the standard $m$-extremal function associated to $E$. More precisely, we have $cap_{m,h^\ast_{m,E,\Omega}}(\Omega)=\int_\Omega(dd^ch^\ast_{m,E,\Omega})^m\wedge\beta^{n-m}=cap_m^\ast(E)$, where the last equality is due to Theorem 1.6.3 in \cite{ref13}. Moreover, assume that $u\in{\mathscr F}^{m}(\Omega)$ and $E\subset\Omega$ such that $cap_{m,u}^\ast(E)=0$, then it is an immediate consequence of the left hand inequality in (2) and Remark 5.3 in \cite{ref12} that $h_{m,E,u}^\ast\equiv 0$. Thanks to the right hand inequality of (1), the same conclusion remains true if we assume that $cap_{m,u}^{\ast}(\overline E)=0$ for a given $u\in {\mathscr E}^{m}(\Omega) $ and $E\Subset\Omega$.  Conversely, if $h_{m,E,u}^\ast\equiv 0$ then the left hand inequality of (1) implies $cap_{m,u}^\ast\left({\ds\mathop  E^\circ}\right)=0$ provided that $u\in {\mathscr E}^{m}(\Omega) $ and $E\Subset\Omega$.
Before presenting the proof of Theorem C, let us recall some further meaningful estimates and applications. We begin with the following estimate due to \cite{Van Nguyen} (see Theorem 3.1):
 \begin{cor} Let $u\in{\mathscr F}^{m}(\Omega)$. Then for any $s>0$, we have $$s^mcap_{m}(\{u<-s\})\leqslant cap_{m,u}^\ast(\{u<-s\})\leqslant  \int_\Omega (dd^cu)^m\wedge\beta^{n-m}.$$
  \end{cor}
  \begin{proof} By monotonicity of $m$-capacity and weighted $m$-capacity and thanks to Theorem 3.1 in \cite{lu article}, we may assume that $u\in{\mathscr E}_0^m(\Omega)\cap{\mathscr C}(\overline\Omega)$. Let $v$ be a negative $m$-sh function satisfying $v\leqslant u$ on $\{u<-s\}$, then $v/s\leqslant u/s<-1$ on $\{u<-s\}$ and consequently we have $v\leqslant sh_{m,\{u<-s\},\Omega}^\ast$. This implies that $h_{m,\{u<-s\},u}^\ast\leqslant sh_{m,\{u<-s\},\Omega}^\ast$. It follows from Corollary 1.4.13 in \cite{ref13}, Proposition 5.2 in \cite{ref12}  and Theorem C that
$$\begin{array}{lcl}s^mcap_{m}(\{u<-s\})&=&s^m\ds\int_\Omega(dd^ch_{m,\{u<-s\},\Omega}^\ast)^m\wedge\beta^{n-m}\\&\leqslant& \ds\int_\Omega(dd^ch_{m,\{u<-s\},u}^\ast)^m\wedge\beta^{n-m}\\&=&cap_{m,u}^\ast(\{u<-s\})\leqslant\ds\int_\Omega(dd^cu)^m\wedge\beta^{n-m}.\end{array}$$
\end{proof}
Assume that $u\in{\mathscr E}_0^m(\Omega)$ and $K$ is a compact set. Starting from the equality $cap_{m,u}(K)=\int_{\Omega}(dd^{c}h_{m,K,u}^{\ast})^{m}\wedge\beta^{n-m}$ (see \cite{Van Nguyen}), and by using Proposition 6 in \cite{ref19}, an adaptation of the proof of Proposition 3.1 in \cite{ref8} yields  that the weighted $m$-capacity $cap_{m,u}$ of sublevel set associated to a given $v\in{\cal F}^m(\Omega)$ decreases at least like $s^{-m}$i.e. there is a constant $A>0$ such that
$$cap_{m,u}\left(\{v<-s\}\right)\leqslant As^{-m},\qquad \forall s>0.$$
Denote by ${\mathscr F}^{m,a}(\Omega)$ the subclass of functions in ${\mathscr F}^{m}(\Omega)$ which have vanishing Hessian measures on $m$-polar sets of $\Omega$. The main consequence of Theorem C is the following characterizations of the Cegrell classes ${\mathscr E}^m(\Omega)$ and ${\mathscr F}^m(\Omega)$ by means of weighted $m$-capacity:
\begin{cor} Let $\Omega $ be a bounded $m$-hyperconvex open subset of $\cb^n,$ $u\in{\mathcal {SH}}_m^-(\Omega)$.\begin{enumerate}
\item Let $E$ be an open set of $\Omega$. If $h_{m,E,u}^\ast\in{\mathscr F}^{m}(\Omega)$ then $cap_{m,h_{m,E,u}^\ast}(E)<+\infty$. Conversely, if $cap_{m,u}(E)<+\infty$ then $h_{m,E,u}^\ast\in{\mathscr F}^{m}(\Omega)$. When $E=\Omega$, we have
$${\mathscr F}^{m}(\Omega)=\left\{u\in{\mathcal {SH}}_m^-(\Omega):\ cap_{m,u}(\Omega)<+\infty\right\}.$$ In particular, if $cap_{m,u}(\Omega)<+\infty$ and $\ds\lim_{s\l +\infty}cap_{m,u}(\{u< -s\})=0$ then $u\in{\mathscr F}^{m,a}(\Omega)$.
\item $u\in{\mathscr E}^m(\Omega)$ if and only if $cap_{m,u}(K)<+\infty$ for every compact set $K$ of $\Omega$.
\end{enumerate}
\end{cor}
Note that the statement concerning ${\mathscr E}^m(\Omega)$ is due to \cite{Van Nguyen}. On the other hand, observe that if $u\in{\mathscr F}^{m}(\Omega)$ and $u$ is $m$-maximal on $\Omega$ then $cap_{m,u}(\Omega)=\int_\Omega(dd^cu)^m\wedge\beta^{n-m}=0$, therefore by Remark 5.3 in \cite{ref12}, we see that $u$ is identically zero. In particular we recover the well-known inclusion in the complex Hessian setting ${\mathscr F}^{m}(\Omega)\subset{\mathscr N}^{m}(\Omega)=\{u\in{\mathscr E}^{m}(\Omega):\ \widetilde u=0\}$, where $\widetilde u$ is the smallest $m$-maximal $m$-sh majorant of $u$ (see \cite{nguyen} for further properties concerning the class ${\mathscr N}^{m}(\Omega)$).
\begin{proof} (1) Assume that $h_{m,E,u}^\ast\in{\mathscr F}^{m}(\Omega)$ for a given open set $E$ of $\Omega$. It is clear from the right hand side inequality of (2) in Theorem C that $$cap_{m,h_{m,E,u}^\ast}(E)=cap_{m,h_{m,E,u}^\ast}^\ast(E)\leqslant\int_\Omega(dd^ch_{m,E,u}^\ast)^m\wedge\beta^{n-m}<+\infty.$$ Conversely, assume that $u\in{\mathcal {SH}}_m^-(\Omega)$ such that $cap_{m,u}(E)<+\infty$. Thanks to Theorem 3.1 in \cite{lu article}, there exists a sequence $(u_j)_j\in{\mathscr E}_0^m(\Omega)\cap{\mathscr C}(\overline\Omega)$ which decreases to $u$. By Proposition 7  we have ${\mathscr E}_0^m(\Omega)\ni h_{m,E,u_j}^\ast\downarrow h_{m,E,u}^\ast$ on $\Omega$. Then using the left-hand side inequality in (2) of Theorem C together with Proposition 9, we get $$\ds\sup_j\int_\Omega(dd^ch_{m,E,u_j}^\ast)^m\wedge\beta^{n-m}\leqslant \ds\sup_jcap_{m,u_j}^\ast(E)=\ds\sup_j cap_{m,u_j}(E)=cap_{m,u}(E)<+\infty.$$
This means in particular that $h_{m,E,u}^\ast\in{\mathscr F}^m(\Omega)$. If $E=\Omega$, we just remark that $h_{m,E,u}= u$ on $\Omega$, which implies that $h_{m,E,u}^\ast = u$ on $\Omega$ because $h_{m,E,u}^\ast=h_{m,E,u}$ a.e on $\Omega$. In particular, if $cap_{m,u}(\Omega)<+\infty$ and $\ds\lim_{s\l +\infty}cap_{m,u}(\{u< -s\})=0$ then we deduce from Corollary 1 that $\ds\lim_{s\l +\infty}s^mcap_{m}(\{u< -s\})=0$ and therefore by Corollary 3.2 in \cite{Van Nguyen} we have $u\in{\mathscr F}^{m,a}(\Omega)$.\\
(2) Assume that $u\in{\mathscr E}^m(\Omega)$ and take an open set ${\mathscr O}$ so that $\mathscr O\Subset\Omega$. From the left hand side inequality of (1) in Theorem C and the fact that $h_{m,\mathscr O,u}^\ast\in{\mathscr F}^{m}(\Omega)$, we have
$$cap_{m,u}(\mathscr O)\leqslant \int_{\Omega}(dd^{c}h_{m,\mathscr O,u}^{\ast})^{m}\wedge\beta^{n-m}<+\infty.$$
Conversely, let $u\in{\mathcal {SH}}_m^-(\Omega)$ such that $cap_{m,u}(K)<+\infty$ for all compact sets $K$ of $\Omega$. Consider an open set $E$ such that $E\Subset\Omega$. By assumption we have $cap_{m,u}(E)<+\infty$ then the first statement yields $h_{m,E,u}^\ast\in{\mathscr F}^m(\Omega)$. Since $h_{m,E,u}^\ast=u$ on $E$ it follows that $u\in{\mathscr E}^{m}(\Omega)$.
\end{proof}
From Corollary 2 if $u\in{\mathscr F}^{m}(\Omega)$ then its total weighted $m$-capacity relatively to $u$, is finite and coincides with the total Hessian measure of $u$. As an application, we consider maximality of measures and minimality of $m$-subharmonic functions by means of the terminology of weighted $m$-capacity. Following \cite{nguyen.v} a measure $\mu$ is said $m$-subharmonically greater than another measure $\nu$ (we write $\mu\succcurlyeq\nu$) if $\int_\Omega -\varphi d\mu\geqslant \int_\Omega -\varphi d\nu$, $\forall \varphi\in{\mathscr E}_0^m(\Omega)\cap{\mathscr C}(\overline\Omega)$. Also a finite measure $\mu$ is said to be maximal if for any measure $\nu$ satisfying $\nu(\Omega)=\mu(\Omega)$, the relation $\nu\succcurlyeq\mu$ implies that $\nu=\mu$. Based on Definition 10 in \cite{nguyen.v} and by Theorem C, a function $u\in{\mathscr F}^m(\Omega)$ is said to be minimal if for any function $v\in{\mathscr F}^m(\Omega)$ the conditions $cap_{m,u}(\Omega)=cap_{m,v}(\Omega)$ and $v\leqslant u$ imply $u=v$.
\begin{cor}\
\begin{enumerate}
\item Let $u\in{\mathscr F}^m(\Omega)$ be such that $(dd^cu)^m\wedge\beta^{n-m}$ is a maximal measure. Then $u$ is minimal.
\item Let $(u_j)_j$ be a sequence in ${\mathscr F}^m(\Omega)$ such that $u_j\downarrow u$ and $\sup_{j}cap_{m,u_j}(\Omega)<+\infty$. Then $u\in{\mathscr F}^m(\Omega)$ and $cap_{m,u_j}(\Omega)\uparrow cap_{m,u}(\Omega)$.
\end{enumerate}
\end{cor}
Corollary 3 is due to \cite{nguyen.v} for the particular case when $cap_{m,u_j}(\Omega)=cap_{m,u_{j+1}}(\Omega)$. However, the requirement $cap_{m,u_j}(\Omega)$ is uniformly finite in (2) cannot be replaced by the weaker assumption that $cap_{m,u_j}(\Omega)$ if finite for each $j$. Indeed, take $u_j=h_{m,{\mathscr B}_0(1-1/j ),{\mathscr B}_0(1)}^\ast$. By Corollary 1 in \cite{ref20} combined with the comments given after Theorem C, we have
$$cap_{m,u_j}(\mathscr B_0(1))=cap_m\left(\mathscr B_0(1-1/j ),\mathscr B_0(1)\right)=\frac{2^n(n-m)^m}{n!m^m\left(\left(1-\frac{1}{j} \right)^{2(1-\frac{n}{m})}-1\right)^m}<+\infty,\quad\forall j.$$
It is clear that $u_j\downarrow -1\not\in{\mathscr F}^m(\mathscr B_0(1))$ (see Proposition 4.5 in \cite{nguyen}) and $cap_{m,u_j}(\Omega)\uparrow +\infty$.
 \begin{proof} (1) Let $v\in{\mathscr F}^m(\Omega)$ such that $cap_{m,u}(\Omega)=cap_{m,v}(\Omega)$ and $v\leqslant u$. By Proposition 5.2 in \cite{ref12}, we see that $(dd^cv)^m\wedge\beta^{n-m}\succcurlyeq(dd^cu)^m\wedge\beta^{n-m}$ then by the maximality hypothesis we have the equality  $(dd^cv)^m\wedge\beta^{n-m}=(dd^cu)^m\wedge\beta^{n-m}$. Applying again Proposition 5.2 in \cite{ref12} to deduce the desired statement.\\
(2) In spite of Theorem 3.1 in \cite{lu article} take a sequence $(w_j)_j\subset{\mathscr E}_0^m(\Omega)\cap{\mathscr C}(\overline\Omega)$ such that $w_j\downarrow u$. For any $j$, $v_j=\max(w_j,u_j)\in{\mathscr E}_0^m(\Omega)$ and clearly $v_j\downarrow u$. Since the weighted $m$-capacity is monotone decreasing with respect to the weight and by Theorem C, we find that
$$\ds\sup_j\int_\Omega (dd^cv_j)^m\wedge\beta^{n-m}=\ds\sup_jcap_{m,v_j}(\Omega)\leqslant \ds\sup_jcap_{m,u_j}(\Omega)<+\infty.$$It follows that $u\in{\mathscr F}^m(\Omega)$ and hence Theorem C implies that $cap_{m,u}(\Omega)<+\infty$. Thus, Proposition 9 completes the proof.  \end{proof}
Now, we are going to prove Theorem C. For this aim, we need the following result which is the corresponding one of Proposition 5.1 of Cegrell \cite {ref7} in the complex Hessian setting:
\begin{pro}
Let $u_p\in {\mathscr F}^{m}(\Omega)$ and $u^j_{p}\in{\mathscr E}_{0}^{m}(\Omega)$ such that $ u^{j}_{p}\downarrow u_{p},\  1\leqslant p\leqslant m $. For every  $h\in {\cal {SH}}_{m}^{-}(\Omega)$, we have
 $$ \lim\limits_{j\rightarrow +\infty} \int_{\Omega}\ds h dd^{c}u_{1}^{j}\wedge...\wedge dd^{c}u_{m}^{j}\wedge\beta^{n-m}=  \int_{\Omega}\ds h dd^{c}u_{1}\wedge...\wedge dd^{c}u_{m}\wedge\beta^{n-m}. $$
\end{pro}
\begin{proofof}{\it Theorem C.} (1) Assume firstly that $u\in {\mathscr E}^{m}_{0}(\Omega)$, then in view of properties of Cegrell classes, for any subset $E\subset\Omega$ since $h_{m,E,u}^{\ast}=\max(h_{m,E,u}^{\ast},u)$, we get $h_{m,E,u}^{\ast}\in{\mathscr E}^m_0(\Omega)$. Assume that $E$ is open such that $E\Subset\Omega$ and let $\{ E_{j}\}_{j \geqslant 1} $ be an exhaustion increasing sequence of compacts subsets of $E$. We have
$$cap_{m,u}(E)=\lim\limits_{j\rightarrow +\infty} cap_{m,u}(E_{j})=\lim\limits_{j\rightarrow +\infty}\int_{\Omega}(dd^{c}h_{m,E_{j},u}^{\ast})^{m}\wedge\beta^{n-m} =\int_\Omega (dd^{c}h^{\ast}_{m,E,u})^{m}\wedge\beta^{n-m} .$$ The second equality is due to \cite{Van Nguyen} because $E_{j}$ is compact while the later one is a consequence of Proposition 11 since $ { \mathscr E_{0}}^{m}(\Omega)\ni h^{\ast}_{m,E_j,u} \downarrow h_{m,E,u}^{\ast}\in {\mathscr F}^{m}(\Omega)$. Assume now that $u\in {\mathscr E}^{m}(\Omega)$, then by Theorem 3.1 in \cite{lu article}, there exists a monotone decreasing sequence $ (u_{j})_j\subset {\mathscr E_{0}}^{m}(\Omega)\cap {\mathscr C}(\overline{\Omega}) $ that decreases to $u$. Thus for each $u_j$, we have
$$cap_{m,u_j}(E)=\int_{\Omega}(dd^{c}h_{m,E,u_j}^{\ast})^{m}\wedge\beta^{n-m}.$$ Thanks to Proposition 7, we have $h_{m,E,u_j}^{\ast}\downarrow h_{m,E,u}^{\ast}\in{\mathscr F}^{m}(\Omega)$. Moreover, by \cite{Van Nguyen} $cap_{m,u}(E)$ is finite, then we apply together Proposition 11 and Proposition 9 when we let $j\l +\infty$ to get
$$cap_{m,u}\left({\ds\mathop  E^\circ}\right)= cap_{m,u}(E)=\int_{\Omega}(dd^{c}h_{m,E,u}^{\ast})^{m}\wedge\beta^{n-m}.$$
Also $cap_{m,u}(E)\leqslant cap_{m,u}(\overline E)$ by monotonicity of $cap_{m,u}$. In order to complete the estimates of (1) for an open subset $E$ of $ \Omega$ such that $E\Subset\Omega$, we just observe that if $K$ is compact then $cap_{m,u}(K)\leqslant cap_{m,u}^\ast(K)$. In fact take an open set ${\mathscr O}$ such that $K\subset{\mathscr O}\Subset\Omega$. So, since $h_{m,K,u}^\ast\geqslant h_{m,{\mathscr O},u}^\ast$ Proposition 5.2 in \cite{ref12} and the preceding argument imply that $$cap_{m,u}(K)=\int_\Omega(dd^ch_{m,K,u}^\ast)^m\wedge\beta^{n-m}\leqslant \int_\Omega(dd^ch_{m,{\mathscr O},u}^\ast)^m\wedge\beta^{n-m}=cap_{m,u}({\mathscr O}).$$
We complete this observation by taking the infimum over all open set ${\mathscr O}\supset K$. Let's now finish the proof and assume that $E$ is a subset satisfying $E\Subset \Omega$ and $u\in{\mathscr E}^{m}(\Omega)$. Observe that $h_{m,{\ds\mathop  E^\circ},u}^\ast\geqslant h_{m,E,u}^\ast\geqslant h_{m,\overline E,u}^\ast$ and by Proposition 7 all these functions are in ${\mathscr F}^{m}(\Omega)$. Once again Proposition 5.2 in \cite{ref12} gives
$$cap_{m,u}\left({\ds\mathop  E^\circ}\right)=\int_{\Omega}\ds (dd^{c}h^{\ast}_{m,{\ds\mathop  E^\circ},u})^{m}\wedge\beta^{n-m}\leqslant \int_{\Omega}\ds (dd^{c}h_{m,E,u}^{\ast})^{m}\wedge\beta^{n-m}\leqslant\ds\int_\Omega (dd^{c}h_{m,\overline E,u}^{\ast})^{m}=cap_{m,u}(\overline E).$$
(2) Let ${\mathscr O}$ be an open set containing $E$ and let $u\in{\mathscr F}^m(\Omega).$ Take a bounded $m$-sh function $v$ on $\Omega$ such that $u\leqslant v\leqslant 0$. Then $\max(u,v)=v\in{\mathscr F}^m(\Omega)$ and therefore by Proposition 5.2 in \cite{ref12}, we have
$$\int_{\mathscr O}(dd^cv)^m\wedge\beta^{n-m}\leqslant\int_\Omega(dd^cv)^m\wedge\beta^{n-m}\leqslant \int_\Omega(dd^cu)^m\wedge\beta^{n-m}.$$
By taking the spremum over $v$ satisfying $u\leqslant v\leqslant 0$, we deduce that
$$cap_{m,u}^\ast(E)\leqslant cap_{m,u}({\mathscr O})\leqslant \int_\Omega(dd^cu)^m\wedge\beta^{n-m}.$$
In order to prove the first inequality in (2) we will take a sequence $(u_j)_j\subset{\mathscr E}_0^m(\Omega)$ such that $u_j\downarrow u$. Observe that $ h_{m,{\mathscr O},u_j}^\ast\leqslant h_{m,E,u_j}^\ast$, for each $j$, then once again Proposition 5.2 in \cite{ref12} implies
$$\int_\Omega(dd^ch_{m,E,u_j}^\ast)^m\wedge\beta^{n-m}\leqslant \int_\Omega(dd^ch_{m,{\mathscr O},u_j}^\ast)^m\wedge\beta^{n-m}=cap_{m,u_j}({\mathscr O})\leqslant cap_{m,u}({\mathscr O}),$$
where the later inequality is obvious because $u\leqslant u_j$. Next, in spite of Proposition 7, we have ${\mathscr E}_0^m(\Omega)\ni h_{m,E,u_j}^\ast\downarrow h_{m,E,u}^\ast\in{\mathscr F}^m(\Omega)$. Thus, a direct application of Proposition 11 gives
$$\int_\Omega(dd^ch_{m,E,u}^\ast)^m\wedge\beta^{n-m}=\ds\lim_{j\l +\infty}\int_\Omega(dd^ch_{m,E,u_j}^\ast)^m\wedge\beta^{n-m}\leqslant cap_{m,u}({\mathscr O}).$$
The required inequality follows by taking the infimum over all open set ${\mathscr O}\supset E$.
\end{proofof}
\begin{rem} From the proof of Theorem C, we have $cap_{m,u}^\ast(E)=\int_{\Omega}(dd^{c}h_{m,E,u}^{\ast})^{m}\wedge\beta^{n-m}$ when $u\in{\mathscr E}^m(\Omega)$ and $E$ is open such that $E\Subset\Omega$. We mention here that such equality was proved by \cite{ref15} (see Theorem 4.1) when $m=n$ and $E\subset\Omega$ is Borel. Unfortunately, the arguments used in their proof are not exact. Indeed, they used Proposition  5.1 of \cite{ref7} without the key assumption that the sequence is decreases as well as the sequence $h_{G_j,u}^\ast$ (as constructed in the proof) is not necessarly increasing to $h_{E,u}^\ast$ on $\Omega$. In complex hessian setting a natural question arises: is the above equality true for  $u\in{\mathscr E}^m(\Omega)$ and $E\Subset\Omega$ similarly with the standard weight $u\equiv -1$?
\end{rem}
\subsection{Sadullaev weighted $m$-capacity} In \cite{ref17},  Sadullaev and Abdullaev introduced the Sadullaev $m$-capacity ${\mathscr P}_{m,\rho}$ in the complex Hessian  case and they compared it with the $m$-capacity $cap_m$. This is a generalization of the study of \cite{ref18} in the complex setting $m=n$. In this part, we introduce the weighted $m$-capacity of Sadullaev $ {\mathscr P}_{m,u,\rho} $ and we establish a relationship with the outer weighted $m$-capacity $cap_{m,u}^\ast,$ provided that the weight $u\in{\mathscr E}^{m}(\Omega)$. Firstly, we state.
\begin{defn}{[\textbf{strongly $m$-pseudoconvex domain}]} Let $ \Omega$ be an open subset of $\cb^{n}$ with ${\mathscr C}^{2}$ boundary. We say that  $\Omega$ is strongly $m$-pseudoconvex if there exists a defining function $\Omega$ (i.e, a $ {\mathscr C}^2$-function $\rho$ such that $\Omega=\{\rho<0\}, \ \rho=0$ and $d\rho\not=0$ on $\partial\Omega$) and there exists $C>0 $ such that at every point of $\Omega$, we have
$(dd^c\rho)^k\wedge\beta^{n-k}\geqslant C\beta^n,\  {\rm for} \ 1\leqslant k\leqslant m.$
\end{defn}
 Now we present the weighted $m$-capacity of Sadullaev ${\mathscr P}_{m,u,\rho}$ as
$${\mathscr P}_{m,u,\rho}(E)=\int_{\Omega}-h_{m,E,u}^{\ast} (dd^{c}\rho)^{m}\wedge\beta^{n-m},$$
for every subset $E\subset\Omega,$ where $\rho$ is the defining function of a strongly $m$-pseudoconvex  domain $\Omega$ and $u\in{\cal {SH}}_m^-(\Omega)$. It is clear that ${\mathscr P}_{m,u,\rho}(E)<+\infty$ if and only if $h_{m,E,u}^{\ast}\in L^1\left((dd^{c}\rho)^{m}\wedge\beta^{n-m}\right)$.  In the special case $u=-1$ we get ${\mathscr P}_{m,-1,\rho}(E)={\mathscr P}_{m,\rho}(E)=\int_{\Omega} -h^{\ast}_{m,E,\Omega}(dd^{c}\rho)^{m}\wedge \beta^{n-m}$. By an argument go back to Sadullaev \cite{ref18} and by investigating a potential theory related to the $m$-capacity, the authors proved in \cite{ref17} that there exist $c_1,c_2>0$ such that $c_1cap_{m}^\ast(E)\leqslant {\mathscr P}_{m,\rho}(E)\leqslant c_2(cap_{m}^\ast(E))^{\frac{1}{m}},$ for every $E\Subset\Omega$, where $cap_m^\ast$ is the standard outer $m$-capacity. Our second main comparison in this section states that:
\vskip0.15cm
\noindent{\bf Theorem D.} {\it Assume that $\Omega$ is a bounded strongly $m$-pseudoconvex domain of $\cb^{n}$ and $u\in{\mathscr E}^{m}(\Omega).$  Then, for every $E\Subset\Omega$  we have
\begin{enumerate}
\item If $\int_{\Omega}u(dd^{c}\rho)^{m}\wedge\beta^{n-m}>-\infty$ then $${\mathscr P}_{m,u,\rho}(E)\leqslant\ds\max_{\Omega}(-\rho)\left[\int_{\Omega}(dd^{c}\rho)^{m}\wedge\beta^{n-m}\right]^{\frac{m-1}{m}} \left(cap_{m,u}^\ast(E)\right)^{\frac{1}{m}}.$$
\item If $E\subset \{\rho<r\}$ for $r<0$ and $\gamma_r=\max\limits_{\{\rho<r\}}\left(-h_{m,E,u}^\ast\right)$ then $$cap_{m,u}^\ast\left({\ds\mathop  E^\circ}\right)\leqslant \frac{m!}{|r|^{m}}\gamma_r^{m-1}{\mathscr P}_{m,u,\rho}(E).$$
\end{enumerate}}
 Theorem D can be seen as a generalization of a result obtained by \cite{ref17} for a weight $u$ identically $-1$ but with the subset $E$ instead of its interior in (2) and with a stronger factor than the one given in \cite{ref17}. Concerning estimate (1), if we assume in addition that $u\leqslant -1$ on $E$, then we get $h_{m,E,\Omega}^\ast\geqslant h_{m,E,u}^\ast$, or equivalently saying that the standard Sadullaev $m$-weighted ${\mathscr P}_{m,\rho}(E)$ can be estimated from above by the weighted $m$-capacity $cap_{m,u}^\ast(E)$.
\begin{proof} (1) By hypothesis and the fact that $u\leqslant h_{m,E,u}^\ast$, we have ${\cal P}_{m,u,\rho}(E)<+\infty$.  Without loss of generality, we may assume that $-1\leqslant\rho\leqslant 0$, then $\rho\in{\mathscr F}^{m}(\Omega)$. Observe also that if ${\mathscr O}$ is an open set satisfying $E\subset{\mathscr O}\Subset\Omega$ then $ h_{m,{\mathscr O},u}^\ast\leqslant h_{m,E,u}^\ast$ and the two functions are in ${\mathcal F}^{m}(\Omega)$. So, an integration by parts on ${\mathscr F}^{m}(\Omega)$  combined with Remark 4 and Proposition  3.2 in \cite{ref12}, yield
$$\begin{array}{lcl}{\cal P}_{m,u,\rho}(E)&=&
\ds\int_{\Omega} -h_{m,E,u}^{\ast} (dd^{c}\rho)^{m}\wedge\beta^{n-m}\\&\leqslant&\ds\int_{\Omega} -h_{m,{\mathscr O},u}^{\ast} (dd^{c}\rho)^{m}\wedge\beta^{n-m}\\&=&\ds\int_{\Omega} -\rho  dd^{c}h_{m,{\mathscr O},u}^{\ast}\wedge (dd^{c}\rho)^{m-1}\wedge \beta^{n-m}\\&\leqslant&\left[\ds\int_{\Omega}-\rho(dd^{c}\rho)^{m}\wedge \beta^{n-m}\right]^{\frac{m-1}{m}}\left[\ds\int_{\Omega}-\rho(dd^{c}h_{m,{\mathscr O},u}^{\ast})^{m}\wedge \beta^{n-m}\right]^{\frac{1}{m}}\\&\leqslant&\ds\max_{\Omega}(-\rho)\left[\ds\int_{\Omega}(dd^{c}\rho)^{m}\wedge\beta^{n-m}\right]^{\frac{m-1}{m}}\left[\ds\int_{\Omega}  (dd^{c}h_{m,{\mathscr O},u}^{\ast})^{m}\wedge \beta^{n-m}\right]^{\frac{1}{m}}= M\left[cap_{m,u}({\mathscr O})\right]^{\frac{1}{m}}, \end{array}$$ where $M={\ds\max_{\Omega}}(-\rho) \left[\int_{\Omega}(dd^{c}\rho)^{m}\wedge\beta^{n-m}\right]^{\frac{m-1}{m}}.$ We complete the proof of (1) by taking the infimum over all
 open set ${\mathscr O}\supset E$. \\
(2) In order to prove the desired inequality, we claim firstly that
If $\sigma=\min\limits_{\Omega} \rho$ and $v\in{\mathscr F}^m(\Omega)$ then for $ \sigma< r < 0$ and $1\leqslant k\leqslant m$ we have
$$
\int_{\sigma}^{r} dt \int_{\{\rho\leqslant t\}} (dd^{c}\rho)^{m-k}\wedge (dd^{c}v)^{k}\wedge \beta^{n-m}\leqslant \max\limits_{\{\rho<r\}}(-v)\int_{\{\rho\leqslant r\}}(dd^{c}\rho)^{m-k+1}\wedge (dd^{c}v)^{k-1}\wedge \beta^{n-m}.
$$
Indeed we may assume that $\max\limits_{\{\rho<r\}}\left(-v\right)<+\infty$, otherwise there is nothing to prove. Suppose also that $v$ is smooth on $\Omega$ then thanks to Stokes and Fubini theorems, we  have
$$\begin{array}{lcl}
 \ds\int_{\sigma}^{r} dt \int_{\{\rho< t\}} (dd^{c}\rho)^{m-k}\wedge (dd^{c}v)^{k}\wedge \beta^{n-m}&=&\ds\int_{\{\rho< r\}} d\rho\wedge d^{c}v\wedge (dd^{c}\rho)^{m-k}\wedge (dd^{c}v)^{k-1}\wedge\beta^{n-m}\\&=&\ds\int_{\{\rho=r\}}v d^{c}\rho\wedge (dd^{c}\rho)^{m-k}\wedge (dd^{c}v)^{k-1}\wedge\beta^{n-m}\\&-&\ds\int_{\{\rho< r\}} v (dd^{c}\rho)^{m-k+1}\wedge (dd^{c}v)^{k-1}\wedge\beta^{n-m}\\&\leqslant&\max\limits_{\{\rho<r\}}\left(-v\right)\ds\int_{\{\rho\leqslant r\}}  (dd^{c}\rho)^{m-k+1}\wedge (dd^{c}v)^{k-1}\wedge\beta^{n-m}. \end{array}$$
The last inequality because $v$ is negative and for any $1\leqslant k\leqslant m$, $d^{c}\rho\wedge (dd^{c}\rho)^{m-k}\wedge (dd^{c}v)^{k-1}\wedge\beta^{n-m}$ define a positive measure on $\{\rho=r\}$ (we argue exactly as in \cite{ref19}). When $v\in{\mathscr F}^m(\Omega)$, we take a sequence $(v_j)_j$ of $m$-sh functions which is smooth and monotone decreasing to $v$ in an open neighbourhood of $\{\rho\leqslant r\}$. Then $$\int_{\sigma}^{r}dt\int_{\{\rho<t\}} (dd^{c}\rho)^{m-k}\wedge (dd^{c}v_j)^{k}\wedge\beta^{n-m}\leqslant \max\limits_{\{\rho<r\}}(-v_j)\int_{\{\rho\leqslant r \}} (dd^{c}\rho)^{m-k+1}\wedge (dd^{c}v_j)^{k-1}\wedge\beta^{n-m}.$$ The proof of the claim was completed by observing that $\max\limits_{\{\rho<r\}}(-v_j)\leqslant \max\limits_{\{\rho<r\}}(-v)$ and by using Lemma 1.2.1 in \cite{ref13} and monotone convergence theorem when we let $j\rightarrow +\infty$. We are going now to finish the proof of the inequality in (2). To this aim let $r<0$ so that $E\subset \{\rho< r\}.$ Hence applying $(m-1)$-times the preceding claim for $v\equiv h_{m,E,u}^{\ast}\in{\mathscr F}^m(\Omega)$ (see Proposition 7), we get
 $$\begin{array}{lcl}& &\ds\int_{\sigma}^{r} dt_{1}\int_{\sigma}^{t_{1}}dt_{2}\cdots\int_{\sigma}^{t_{m-1}} dt_{m}\int_{\{\rho\leqslant t_{m}\}}(dd^{c}h_{m,E,u}^{\ast})^{m}\wedge\beta^{n-m}\leqslant\\&\leqslant&\gamma_r^{m-1}\ds\int_{\sigma}^{r}dt_{1}\int_{\{\rho\leqslant t_{1}\}}dd^{c}h_{m,E,u}^{\ast}\wedge (dd^{c}\rho)^{m-1}\wedge \beta^{n-m}.\end{array}$$ It is obvious that the left hand side integral is bounded from below by  $$\frac{|r|^{m}}{m!}\int_{\{\rho\leqslant r\}} (dd^{c}h_{m,E,u}^{\ast})^{m}\wedge \beta^{n-m}. $$ We may suppose that ${\mathscr P}_{m,u,\rho}(E)<+\infty$ and $\gamma_r<+\infty$ otherwise there is nothing to prove in (2). So, concerning the right hand integral, as in the proof of the claim, Fubini theorem yields
 $$\begin{array}{lcl}\ds\int_{\sigma}^{r}dt_{1}\int_{\{\rho\leqslant t_{1}\}}dd^{c}h_{m,E,u}^{\ast}\wedge (dd^{c}\rho)^{m-1}\wedge \beta^{n-m}&\leqslant&{\mathscr P}_{m,u,\rho}(E)+\ds\int_{\{\rho=r\}}h_{m,E,u}^{\ast}d^c\rho\wedge (dd^{c}\rho)^{m-1}\wedge \beta^{n-m}\\&\leqslant&{\mathscr P}_{m,u,\rho}(E)\end{array}$$ because  $h_{m,E,u}^\ast\leqslant 0$ and $d^c\rho\wedge (dd^{c}\rho)^{m-1}\wedge \beta^{n-m}$ define a positive measure on $\{\rho=r\}$ (see \cite{ref19}). Finally, thanks to Proposition 7 the measure $(dd^{c}h_{m,E,u}^{\ast})^{m}\wedge\beta^{n-m}$ is supported by $\overline E\subset\{\rho\leqslant r\}$, then in view of Theorem C, we deduce that
$$ cap_{m,u}^{\ast}\left({\ds\mathop  E^\circ}\right)=cap_{m,u}\left({\ds\mathop  E^\circ}\right)\leqslant\int_{\{\rho\leqslant r\}}(dd^{c}h_{m,E,u}^{\ast})^{m}\wedge\beta^{n-m} \leqslant \gamma_r^{m-1}\frac{m!}{|r|^{m}} {\mathscr P}_{m,u,\rho}(E).$$\end{proof}
\section{$(m,T)$-Complex Hessian equation}
Assume that $\mu$ be a positive measure on a bounded open subset $\Omega$ of $\cb^n$ and let $ T\in{\mathscr C}_p^m(\Omega)$, $m\geqslant  p+1$.  Inspired by an argument due to Xing \cite{ref21}, we investigate the following complex Hessian equation associated to $T$: $ (dd^c v)^{m-p}\wedge\beta^{n-m}\wedge T =\mu$. It is worth pointing out that the case $T=1$ leads to the famous complex Hessian equation which was the subject of many papers the later fiveteen years, see for example \cite{ref4},\cite{Den},\cite{ref12},\cite{ref13}, etc. We start with the following Lemma, which is an easy adaptation of the proof of Corollary 7.3 in \cite{ref2} (see \cite{ref13}):
\begin{lem}
Let $(u_{j})_j\subset{\cal {SH}}_{m}(\Omega)$ be locally uniformly bounded from above, $u=\limsup u_{j}$ Then $u^\star\in{\cal {SH}}_{m}(\Omega)$ and $\{u<u^\star \}$ is $m$-polar.
\end{lem}
Denotes by $\|\mu\|_E$ the mass on $E$ of the total variation of a signed measure $\mu$. We also  get the following Xing type inequality which is true in the complex Hessian context.
\begin{lem}
Let $ \Omega$ be a bounded open subset of $ \cb^{n} $, $ T\in{\mathscr C}_p^m(\Omega)$ and $ u,v\in {\cal {SH}}_{m}(\Omega)\cap L^{\infty}_{loc}(\Omega) $ such that $ \lim\limits_{\xi\rightarrow\partial\Omega\cap\Supp T}\sup|u(\xi)-v(\xi)|=0. $ Then, $ \forall\delta>0 $ and $ 0<k<1,$ we have
$$ cap_{m,T}\left(|u-v|\geqslant\delta\right)\leqslant {(m-p)!\over
(1-k)^{m-p}\delta^{m-p}} \| (dd^c u)^{m-p}\wedge\beta^{n-m}\wedge T - (dd^c
v)^{m-p}\wedge\beta^{n-m}\wedge T\|_{\{|u-v|>k\delta \}}.$$
In particular, if
$(dd^c u)^{m-p}\wedge\beta^{n-m}\wedge T=(dd^c v)^{m-p}\wedge\beta^{n-m}\wedge T$ then $u=v\ cap_{m,T}$-a.e.
\end{lem}
\begin{rem}\ \begin{enumerate}
\item When $T=1$, we recover Corollary 1.3.15 of Lu \cite{ref13} as well as the estimate obtained by Xing \cite{ref21} when $T=1$ and $m=n$.
\item Suppose that $0\in\Omega$ and $L$ is a complex subspace of codimension $p$ in $\cb^n$. Consider the current of integration $T=[L]$ and assume that $ \lim\limits_{\xi\rightarrow\partial\Omega\cap L}\sup|u(\xi)-v(\xi)|=0$ and $(dd^c (i^{\star} u))^{m-p}\wedge (i^{\star}\beta)^{n-m}=(dd^c (i^{\star}v))^{m-p}\wedge (i^{\star}\beta)^{n-m}$, where $i ; L\cap\Omega\hookrightarrow\Omega$ is the injection map. In view of Remark 4  in \cite{ref9} and as a consequence of Lemma 4, we get $u_{| \Omega\cap L}=v_{| \Omega\cap L}$. We recover then Corollary 1.3.15 in \cite{ref13} for the particular case $L=\cb^n$.
\item Assume that $ u_j,u\in {\cal {SH}}_{m}(\Omega)\cap L^{\infty}_{loc}(\Omega) $ such that $ \lim\limits_{\xi\rightarrow\partial\Omega\cap\Supp T}\sup|u_j(\xi)-u(\xi)|=0 $ uniformly in $j$ and
$\| (dd^c u_j)^{m-p}\wedge\beta^{n-m}\wedge T - (dd^c
u)^{m-p}\wedge\beta^{n-m}\wedge T\|_E\l 0$ for all $E\Subset\Omega$. Then, it is clear from Lemma 4 that $u_j$ converges to $u$ in capacity $cap_{m,T}$ on $\Omega$.
\end{enumerate}
\end{rem}
\begin{proof}
Let $w\in{\cal {SH}}_m(\Omega,[0,1]),\delta >0$ and $k\in{}]0,1[$. Thanks to Lemma 3  in \cite{ref9} and the fact that $\{|u-v|\geqslant \delta \}\subset \{|u-v\pm\delta k|\geqslant (1-k)\delta \}$, one get:
$$\begin{array}{lcl}& &\ds\int_{\{|u-v|\geqslant \delta \}} (dd^c w)^{m-p}\wedge\beta^{n-m}\wedge T\leqslant\\ &\leqslant&\ds{1\over (1-k)^{m-p}
\delta^{m-p}}\ds\int_{\{u+\delta\leqslant v\}}(v-u-k\delta )^{m-p}(dd^c w)^{m-p}\wedge\beta^{n-m}\wedge T \\ &+&\ds{1\over
(1-k)^{m-p}\delta^{m-p}}\ds\int_{\{v+\delta\leqslant u\}}(u-v-k\delta
)^{m-p}(dd^c w)^{m-p}\wedge\beta^{n-m}\wedge T\hfill\cr &\leqslant&\ds{1\over
(1-k)^{m-p}\delta^{m-p}}\ds\int_{\{u+k\delta <v\}}(v-u-k\delta
)^{m-p}(dd^c w)^{m-p}\wedge\beta^{n-m}\wedge T \\&+&\ds{1\over
(1-k)^{m-p}\delta^{m-p}}\ds\int_{\{v+k\delta <u\}}(u-v-k\delta
)^{m-p}(dd^c w)^{m-p}\wedge\beta^{n-m}\wedge T\\&\leqslant&\ds{(m-p)!\over
(1-k)^{m-p}\delta^{m-p}}\ds\int_{\{|u-v|>k\delta \}}
(1-w)\left(\chi_{\{u+k\delta <v\}}-\chi_ {\{v+k\delta <u\}}\right)(dd^c u)^{m-p}\wedge\beta^{n-m}\wedge T\\&-&\ds{(m-p)!\over
(1-k)^{m-p}\delta^{m-p}}\ds\int_{\{|u-v|>k\delta \}}
(1-w)\left(\chi_{\{u+k\delta <v\}}-\chi_ {\{v+k\delta <u\}}\right)(dd^c v)^{m-p}\wedge\beta^{n-m}\wedge T\\&\leqslant&\ds{(m-p)!\over (1-k)^{m-p}\delta^{m-p}}
\|(dd^c u)^{m-p}\wedge\beta^{n-m}\wedge T-(dd^c v)^{m-p}\wedge\beta^{n-m}\wedge T\|_{\{|u-v|>k\delta \}}.\end{array}
$$
The proof was completed by arbitrariness of $w$.
\end{proof}
As we mention at the beginning of this section, we deal with the complex Hessian equation associated to a given $m$-positive closed current $T$. Since $T$  may have too strong singularities, we direct ourselves to the subclass ${\mathscr A}_m^p(\Omega)$ of currents $R\in{\mathscr C}_p^m(\Omega)$ such that every $R$-negligible subset (i.e, vanishing with respect to the trace measure $\sigma_R$) is Lebesgue negligible. Note that ${\mathscr A}_m^p(\Omega)$ contains at least currents like $(dd^cv)^p$, where $v$ is a ${\mathcal C}^2$-strictly $m$-sh functions i.e. $v$ is of classe ${\mathcal C}^2$ and for every $\varphi\in{\mathscr D}(\Omega)$ there is $\varepsilon>0$ such that $v+\varepsilon\varphi$ is $m$-sh. Now, we are ready to state our main subsolution Theorem in this section.
\vskip0.15cm
\noindent{\bf Theorem E.} {\it Let $\Omega$ be an open bounded subset of $\mathbb{C}^n$, $T\in {\mathscr A}_p^m(\Omega)$, $m\geqslant p+1$ and $\mu$ is a positive measure on $\Omega$ satisfying:
\begin{enumerate}
\item There exists $v\in{\cal {SH}}_{m}(\Omega)\cap L^{\infty }(\Omega)$ such that $ (dd^c v)^{m-p}\wedge\beta^{n-m}\wedge T \geqslant \mu.$
\item There exists a sequence of measures $\mu_{j}=(dd^c u_{j})^{m-p}\wedge\beta^{n-m}\wedge T$
such that $\|\mu_{j}-\mu \|_{\Omega}\l 0$  where $u_{j}\in
{\cal {SH}}_{m}(\Omega)\cap {\mathscr C}(\overline {\Omega})\  and \ u_{j}=u_{1}$ on $\partial
\Omega\cap\Supp T$ for every $ $j. Assume moreover that for any $j$, $cap_{m,T}(v_{j}<v_{j}^{\ast})=0,$ where $v_{j}=\sup\{u_{k}, k\geqslant j\}$.
\end{enumerate}
Then there exists $u\in
{\cal {SH}}_{m}(\Omega)\cap L^{\infty }(\Omega)$ such that $\mu=  (dd^c u)^{m-p}\wedge\beta^{n-m}\wedge T.$}
\vskip0.1cm
Theorem E generalizes a result due to Xing \cite{ref21} when $T=1$ and  $m=n$.  According to \cite{subsolution}, when $T=\beta^{p}$ and $\Omega $ is a smoothly strongly pseudoconvex domain, the second hypothesis of Theorem E is superfluous. In fact the author has proved that the complex Hessian equation $ (dd^c v)^{m-p}\wedge\beta^{n-m+p}=\mu$  admits a solution $u\in \cal {SH}_{m-p}(\Omega )\cap L^{\infty}(\Omega)$ provided that it has a subsolution. In our situation, $\Omega$ is only bounded, then together with the second hypothesis, we obtain a more precise solution $u\in \cal {SH}_{m}(\Omega)\cap L^{\infty}(\Omega).$
\begin{proof} The proof follows the same reasoning as in \cite{ref21}. Let $A>0$ such that $\forall z\in \overline {\Omega},$ $A\geqslant |z|$ and select $c>0$ so that $c\geq |v(z)|+|u_{1}(w)|+1$ $\forall z\in \Supp T$ and $\forall w\in \partial\Omega\cap\Supp T$. By applying Lemma 3 in \cite{ref9} for  $r=1, w_{1}=...=w_{m-p}=\frac{|z|^{2}} {A^{2}}$ and in view of hypothesis (1), we have:
$$
\begin{array}{lcl}\ds\int_{\{u_{j}<v-c\}}\left(1-{|z|^{2}\over A^{2}}\right)
 (dd^c u_{j})^{m-p}\wedge\beta^{n-m}\wedge T
&\geqslant&\ds\int_{\{u_{j}<v-c\}}\left(1-{|z|^{2}\over A^{2}}\right) (dd^c v)^{m-p}\wedge\beta^{n-m}\wedge T
\\&+&\ds{1\over (m-p)!A^{2(m-p)}}\ds\int_{\{u_{j}<v-c\}}(v-c-u_{j})^{m-p} \beta^{n-p}\wedge T
\\&\geqslant&\ds\int_{\{u_{j}<v-c\}}\left(1-{|z|^{2}\over A^{2}}\right)d\mu
\\&+&\ds{1\over (m-p)!A^{2(m-p)}}\ds\int_{\{u_{j}<v-c\}}(v-c-u_{j})^{m-p} \beta^{n-p}\wedge T
\end{array}
$$ Moreover, since $\|\mu_{j}-\mu \|_{\Omega}\l 0$ then by Fatou Lemma, we get:
$$
\begin{array}{lcl}0&\geqslant& \ds\frac{1}{(m-p)!A^{2(m-p)}}\ds{\ds\mathop{{\underline{\rm lim}}}_{j\l+\infty}}
\ds\int_{\{u_{j}<v-c\}}(v-c-u_{j})^{m-p} \beta^{n-p}\wedge T\\&\geqslant&\ds\frac{1}{(m-p)!A^{2(m-p)}}\ds\int_{\Omega}\ds{\ds\mathop{{\underline{\rm lim}}}_{j\l+\infty}}
\left(\chi_{\{u_{j}<v-c\}}(v-c-u_{j})^{m-p}\right)\beta^{n-p}\wedge T \\&\geqslant&\ds\frac{1}{(m-p)!A^{2(m-p)}}\ds\int_{\Omega}\chi_{\left\{\ds{\ds\mathop{{\overline{\rm lim}}}_{j\l+\infty}}
u_{j}<v-c\right\}}\left(\ds{\ds\mathop{{\underline{\rm lim}}}_{j\l+\infty}}|v-c-u_{j}|\right)^{m-p}
 \beta^{n-p}\wedge T\\&\geqslant&\ds\frac{1}{(m-p)!A^{2(m-p)}}\ds\int_{\Omega}\chi_{\left\{\ds{\ds\mathop{{\overline{\rm lim}}}_{j\l+\infty}}
u_{j}<v-c\right\}}\left(v-c-\ds{\ds\mathop{{\overline{\rm lim}}}_{j\l+\infty}}u_{j} \right)^{m-p}
 \beta^{n-p}\wedge T.
\end{array}$$
It follows that $\ds\mathop{{\overline{\rm lim}}}_{j\l+\infty}u_{j}> v-c$ for $\sigma_T$-a.e and therefore $\ds{\ds\mathop{{\overline{\rm lim}}}_{j\l+\infty}}u_{j}
\not \equiv -\infty $. Setting $A=\cup_j (v_{j}<v_{j}^{\ast}).$ In spite of Lemma 3, there exists $g\in {\cal {SH}}_{m}(\Omega)$ such that:
$$
v_{j}=v_{j}^{\ast}\downarrow
\ds{\ds\mathop{{\overline{\rm lim}}}_{j\l+\infty}}u_{j}=g\ \ \  {\rm sur}\ \Omega\smallsetminus A.
$$ Moreover, since $A$ is $(m,T)$-pluripolar and $T\in {\mathscr A}_m^p$ then $g\geqslant v-c$ a.e. Hence, $g$ is locally bounded on  $\Omega.$ Now, we are going to prove that $u_{j}$ converges to $g$ in $ cap_{m,T}$ on every $E\Subset\Omega.$ For this aim, let $E\Subset \Omega, $ then for all $ \delta>0,$ we have:
\begin{eqnarray}\label{e1}
&&cap_{m,T}\left(E\cap \{|g-u_{j}|\geq\delta\}\right)\leqslant cap_{m,T}\left(E\cap \{|g-v_{j}|\geqslant {\delta\over 2}
\right) +cap_{m,T}\left(\{|v_{j}-u_{j}|\geq
{\delta\over 2}\}\right).
\end{eqnarray}
Thanks to the quasicontinuity of  $g$ and the Dini's Theorem, it is not hard to see that $v_{j}\downarrow g$ uniformly on $\Omega\smallsetminus {\mathscr O}$ where ${\mathscr O}$ is an open subset of $\Omega $ with arbitrarily small $(m,T)$-capacity. Hence, the first term in the right tends to $0$ when $j\l +\infty.$ For the second one, we show firstly that:
\begin{eqnarray}\label{e2}
B=\left\{|v_{j}-u_{j}|\geqslant \delta/ 2\right\}\subset
\ds\cup_{l=0}^{+\infty }\left\{|u_{j+l+1}-u_{l+j}|\geqslant \delta/{2^{l+j+2}}\right\}.
\end{eqnarray}
Let $z_0\in B,$ $l_0$ such that $|u_{j+l_{0}+1}(z_{0})-u_{j}(z_0)|\geqslant
\delta/4$ and assume that:
$$
z_0\not\in\ds\cup_{l=0}^{l_{0}-1}\left\{|u_{j+l+1}-u_{l+j}|\geq \delta/{2^
{l+j+2}}\right\}.
$$
So,$$
\begin{array}{lcl}|u_{j+l_{0}+1}(z_{0})-u_{l_{0}+j}(z_{0})|
&\geqslant&|u_{j+l_{0}+1}(z_{0})-u_{j}(z_{0})|
-\ds\sum_{l=0}^{l_{0}-1}|u_{j+l+1}(z_{0})-u_{l+j}(z_{0})|\\&\geqslant&\delta/4-\ds\sum_{l=0}^{l_{0}-1}\delta/{2^{l+j+2}}\geqslant
\delta/{ 2^{l_0+j+2}}.\end{array}
$$
It follows that   $z_0 \in \left\{|u_{l_0+j+1}-u_{l_0+j}|\geq \delta/{2^{l_{0}+j+2}}\right\},$ which proves $(\ref{e2})$. Moreover, by passing to an extracted subsequence, which denoted also by $\mu_{j}$, we may assume that for any $j$, we have $\ (m-p)!\|\mu_{j}-\mu\|_{\Omega}\leqslant  2^{-(m-p+1)j}.$ Therefore by applying Lemma 4, we get:
$$
\begin{array}{lcl}cap_{m,T}\left\{|u_{j+1}-u_{j}|>\delta\right\}\leqslant \ds{(m-p)!\over \delta^{m-p}}
\|\mu_{j+1}-\mu_{j}\|_{\Omega}
&\leqslant& \ds{(m-p)!\over \delta^{m-p}} (\|\mu_{j+1}-\mu \|_{\Omega}+\|\mu-
\mu_{j}\|_{\Omega})\\
&\leqslant& \ds{2\over \delta^{m-p}2^{(m-p+1)j}}.\end{array}$$
Consequently,
$$
\begin{array}{lcl}cap_{m,T}\left\{|{\rm sup}\{u_{j},u_{j+1},...\}-u_{j}|\geq \delta/ 2\right\}
&\leqslant&\ds\sum_{l=0}^{+\infty}cap_{m,T}\left\{|u_{l+j+1}-u_{l+j}|\geqslant\delta/{2^{l+j+2}}\right\}\\&\leqslant&\ds\sum_{l=0}^{+\infty}{ 2\ 2^{(m-p)(l+j+2)}\over\delta^{m-p}2^{(m-p+1)(l+j)}}
={4^{m-p+1} \over\delta^{m-p} 2^j}.\end{array}
$$ We see then that the second term in the right  side of $(\ref{e1})$ tends to $0$ when $j\l +\infty.$ Finally, by using Theorem 3 in \cite{ref9}, we see that $(dd^{c}u_{j})^{m-p}\wedge \beta^{n-m}\wedge T$ converges weakly to the measure $(dd^{c}g)^{m-p}\wedge \beta^{n-m}\wedge T$ which completes the proof.
\end{proof}
\section*{Acknowledgments}
 The first named author wishes to thank Professor Ahmed Zeriahi for the hospitality and helpful discussions during her visit to "Institut de math\'ematiques de Toulouse". Also the authors would like to thank the referee for his/her remarks and suggestions which helped to improve the presentation of the paper.

 \section*{Data availability statement}
 Data sharing not applicable to this article as no datasets were generated or analysed during
the current study.


\begin{thebibliography}
{X-XX1}
\bibitem[1]{ref1}\textbf{Andersson M., Blocki Z. and Wulcan E.}, On a Monge-Amp\`ere operator for plurisubharmonic functions with analytic singularities. Indiana. Univ. Math. J, 68 (2019): 1217-1231.
\bibitem[2]{ref2} \textbf{Bedford E. and Taylor B. A.}, A new capacity for plurisubharmonic functions. Acta Math. 149, no 1 (1982): 1-40.
\bibitem[3]{ref3}\textbf{Bedford E.}, Survey of pluripotential theory, Several Complex Variables,
Mittag-Leffler Institute 1987-88, Math. Notes, 38 (1993): 48-95.
\bibitem[4]{ref4} \textbf{Benelkourchi S.}, Weighted pluricomplex energy. Potential. Anal, 31, no 1 (2009): 1-20.
\bibitem[5]{me-el}\textbf{Ben Messaoud H.  and El Mir H.}, Op\'erateur de Monge-Amp\`ere et tranchage des courants positifs ferm\'es. J. Geo. Anal, 10, no 1 (2000): 139-168.
\bibitem[6]{Bl}\textbf{Blocki Z.}, Weak solutions to the complex Hessian equation. Ann. Inst. Fourier, Grenoble, 55, no 5 (2005): 1735-1756.
  \bibitem[7]{ref6} \textbf{Cegrell U.}, Capacities in Complex Analysis, {\sl Braunschwerg Wiesbaden Friedr
Vieweg et Sohn}, (1988).
\bibitem[8]{ref7} \textbf{Cegrell U.}, The general definition of the complex Monge-Amp\`ere operator. Ann. Inst. Fourier, Grenoble, no 1  (2004): 159-179.
 \bibitem[9]{ref8} \textbf{Cegrell U., Kolodziej S. and Zeriahi A.}, Subextension of plurisubharmonic functions with weak singularities. Math. Zeit, 250, no 1  (2005): 7-22.
     \bibitem[10]{ref10} \textbf{Demailly J.-P.}, Complex analytic and differential geometry. Preprint (2009).
http://www-fourier.ujf-grenoble.fr/demailly/manuscripts/agbook.pdf.
\bibitem[11]{Den} \textbf{Dinew S. and Kolodziej S.}, A priori estimates for the complex Hessian equation. Anal. PDE,7, no 1 (2014), 227-244.
\bibitem[12]{ref9} \textbf{Dhouib A. and Elkhadhra F.}, $m$-Potential theory associated to a positive closed current in the class of $m$-sh functions. Complex. Var. Elliptic. Equ,  61, no 7 (2016); 875-901.
\bibitem[13]{ref20} \textbf{Elkhadhra F.}, $m$-Generalized Lelong Numbers and Capacity Associated to a Class of $m$-Positive Closed Currents. Result. Math,  74, no 1 (2019): 01-24.
 \bibitem[14]{ref12} \textbf{Hung V.V. and Phu N.V.}, Hessian measures on $m$-polar sets and applications to the complex Hessian equations. Complex. Var. Elliptic. Equ, 62, no 8 (2017): 1135-1164.
  \bibitem[15]{ref15} \textbf{Le Mau H., Nguyen V. and Pham H.}, The weighted relative extremal functions and weighted capacity.  Acta. Math. Viet, 31, no 3  (2006): 219-230.
\bibitem[16]{ref13} \textbf{Lu C.H.}, Complex Hessian equations. Doctoral thesis, University of Toulouse
III Paul Sabatier (2012).
\bibitem[17]{lu article}\textbf{Lu C.H.}, A variational approach to complex Hessian equations in $\cb^n$. J. Math. Anal. Appl, 431, no 1 (2015): 228-259.
\bibitem[18]{subsolution}\textbf{Nguyen, N. C.}, Subsolution theorem  for the complex Hessian equation.  Univ. Lag. Acta. Math, Fasciculus L (2012): 69-88.
\bibitem[19]{ref16}\textbf{Nguyen V.K. and Pham H.H.}, A comparison principle for the complex Monge-Amp\`ere operator in Cegrell classes and applications. Trans. Amer. Math. Soc, 361, no 10 (2009): 5539-5554.
\bibitem[20]{Van Nguyen}\textbf{Nguyen V.T.}, A characterization of the Cegrell classes and generalized $m$-capacities. Ann. Polon. Math, 121, no 1 (2018):  33-43.
    \bibitem[21]{nguyen.v}\textbf{Nguyen V.T.}, On $m$-Subharmonic Ordering of Measures. Result. Math, 73, (2018): 1-18.
   \bibitem[22]{nguyen}\textbf{Nguyen V.T.}, Maximal $m$-subharmonic functions and the Cegrell class ${\mathscr N}_m$. Indag. Math. New series, 30, no 4 (2019): 717-739.
  \bibitem[23]{ref17} \textbf{Sadullaev A. and Abdullaev B.}, Potential theory in the class of $m$-subharmonic functions. Proc. Stek. Inst. Math, 279, no 1 (2012): 155-180.
\bibitem[24]{ref18} \textbf{Sadullaev A.},  Plurisubharmonic Functions.  In Several complex variables II,  Springer, Berlin, Heidelberg,  (1994): 59-106.
  \bibitem[25]{ref19}\textbf{Wan D. and Wang W.},  Complex Hessian operator and Lelong number for unbounded $m$-subharmonic functions. Potential. Anal, 44 (2016): 53-69.
   \bibitem[26]{ref21}\textbf{Xing Y.}, Continuity of the complex Monge-Amp\`ere operator. Proc. Amer. Math. Soc, 124, no 2 (1996): 457-467.


   \end{thebibliography}
  \end{document}